\documentclass[10 pt,reqno]{amsart}

\usepackage{amsmath}
\usepackage{amssymb}
\usepackage{amsfonts}
\usepackage{amsthm}
\usepackage{hyperref}
\usepackage{tikz}
\usepackage{enumerate}

\usepackage{bbm}
\usepackage[mathscr]{eucal}

\newcommand{\Be}{\begin{equation}}
\newcommand{\Ee}{\end{equation}}
\newcommand{\Bea}{\begin{align}}
\newcommand{\Eea}{\end{align}}
\newcommand{\Beas}{\begin{align*}}
\newcommand{\Eeas}{\end{endalign*}}
\newcommand{\Benu}{\begin{enumerate}}
\newcommand{\Eenu}{\end{enumerate}}
\newcommand{\Bi}{\begin{itemize}}
\newcommand{\Ei}{\end{itemize}}

\theoremstyle{plain}
\newtheorem{thm}{Theorem}[section]
\newtheorem{lemma}[thm]{Lemma}
\newtheorem{cor}[thm]{Corollary}
\newtheorem{proposition}[thm]{Proposition}

\theoremstyle{remark}

\newtheorem*{remarka}{Remark}

\numberwithin{equation}{section}

\makeatletter
    \def\@@and{}
\makeatother

\begin{document}

\title[Spherical maximal operators on radial functions]{Spherical maximal operators with fractal sets of dilations on radial functions}

\author[D. Beltran]{David Beltran}
\address{Departament d'Anàlisi Matemàtica, Universitat de València, Burjassot, Spain}
\email{david.beltran@uv.es}

\author[J. Roos]{Joris Roos}
\address{Department of Mathematics and Statistics, University of Massachusetts Lowell, Lowell, MA, USA}
\email{joris\_roos@uml.edu}

\author[A. Seeger]{Andreas Seeger}
\address{Department of Mathematics, University of Wisconsin--Madison, Madison, WI, USA}
\email{seeger@math.wisc.edu}

\subjclass[2020]{42B25, 28A80}
\keywords{
Spherical maximal functions, $L^p$ improving estimates, radial functions, Minkowski and Assouad dimensions, Assouad spectrum}

\begin{abstract}
For a given set of dilations $E\subset [1,2]$, Lebesgue space mapping properties of the spherical maximal operator with dilations restricted to $E$ are studied when acting on radial functions. In higher dimensions, the type set only depends on the upper Minkowski dimension of $E$, and in this case complete endpoint results are obtained. In two dimensions we determine the closure of the $L^p\to L^q$ type set for every given set $E$ in terms of a dimensional spectrum closely related to the upper Assouad spectrum of $E$.
\end{abstract}

\maketitle

\section{Introduction}

For $d \geq 2$ and $f \in L^1_{\mathrm{loc}}({\mathbb R}^d)$, define $A_t f(x)$ as the average of $f$ over a sphere of radius $t$ centered at $x\in\mathbb{R}^d$.
Given a set of dilations $E\subset [1,2]$ the spherical maximal operator $M_E$ is given by
\[ M_E f = \sup_{t\in E} |A_t f(x)|. \]
Let
$L^p_{\mathrm{rad}}\subset L^p$ denote the space of radial $L^p$ functions $f$, that is $f$ taking the form $f(x)=f_0(|x|)$.
In this paper we are interested in the {\em radial type set} of $M_E$, that is
\[ \mathcal{T}_E^{\mathrm{rad}} = \{ (\tfrac1p,\tfrac1q)\in [0,1]^2\,:\, M_E:L^p_{\mathrm{rad}}\to L^q \}. \]
With $\mathcal{T}_E$ denoting the (full) type set of exponent pairs $(\frac1p,\frac1q)\in [0,1]^2$ such that $M_E$ is bounded $L^p\to L^q$ it is clear that $\mathcal{T}_E$ is contained in
$\mathcal{T}_E^{\mathrm{rad}}$. This inclusion is typically strict.

Bounds for spherical maximal functions have been studied extensively in the literature, starting with the work of Stein \cite{SteinPNAS1976} and Bourgain \cite{BourgainJdA1986} in the $q=p$ case and of Schlag \cite{Schlag1997} and Schlag and Sogge \cite{SchlagSogge1997} in the $q > p$ case, whenever $E=[1,2]$; see also the work of Leckband \cite{Leckband} and, more recently, by Nowak, Roncal and Szarek \cite{NowakRoncalSzarek} in the case of radial functions. The case of restricted sets of dilations $E \subset [1,2]$ was first explored by Wainger and Wright and one of the authors \cite{SeegerWaingerWright1995}, and continued in \cite{SeegerWaingerWright1997}, \cite{STW-jussieu} in the $q=p$ case. For the case $q > p$, two of the authors \cite{RoosSeeger} described the class of closed convex sets that may arise as $\overline{\mathcal{T}_E}$ for some $E\subset [1,2]$. Moreover,
for large classes of sets $E$, the shape of $\overline{\mathcal{T}_E}$ was determined in \cite{AHRS}, \cite{RoosSeeger}. However, it is currently not known how to determine $\overline{\mathcal{T}_E}$ for {\it general } subsets $E\subset [1,2]$.

In this paper we solve this problem for radial functions, by describing $\overline{\mathcal{T}_E^{\mathrm{rad}}}$ for all dilation sets $E\subset [1,2]$. When $d\ge 3$ we also settle all endpoint cases,
and thus determine $\mathcal{T}_E^{\mathrm{rad}}$ for all $E\subset [1,2]$.

For $d\ge 3$, the set $\overline{{{\mathcal{T}}}_E^{\mathrm{rad}}}$ is a closed triangle depending on the upper Minkowski dimension of $E$ (Theorem \ref{thm:highdim}).
In the more interesting case $d=2$, the shape of $\overline{{{\mathcal{T}}}_E^{\mathrm{rad}}}$ is not necessarily polygonal and we give a closed formula in terms of a spectrum of dimensional quantities closely related to the upper Assouad spectrum of $E$ (Theorem \ref{thm:maintwodim}).

\subsubsection*{The case $d\ge 3$.} We begin by describing the more elementary result on $L^p_{{\text{\rm rad}}}\to L^q$ boundedness for $d\ge 3$. For $\delta<1$ let $N(E,\delta) $ denote the minimum number of intervals of length $\delta$ required to cover $E$. Let $\beta$ be the {\it upper Minkowski dimension} of $E$, defined by
\[ \beta=\dim_{\mathrm M}E= \limsup_{\delta\to 0} \frac{\log N(E,\delta)}{\log \tfrac 1\delta }.\]
We denote by $\Delta_\beta=\triangle (P_1, P_2, P_3)$ the closed triangle with vertices
\begin{equation*}
P_1= (0,0), \quad P_{2,\beta} =(\tfrac{d-1}{d-1+\beta}, \tfrac{d-1}{d-1+\beta}), \quad P_{3,\beta}^{\mathrm{rad}}:=(\tfrac{d(d-1)}{d^2-1+\beta}, \tfrac{d-1}{d^2-1+\beta}).
\end{equation*}
\begin{thm}\label{thm:highdim}
Let $d\ge 3$ and $E\subset [1,2]$. Then $\overline{\mathcal{T}_E^{\mathrm{rad}}} = \Delta_\beta$.
More precisely:
\begin{enumerate}[(i)]
\item If $\beta<1$ and $\sup_{\delta<1} \delta^\beta N(E,\delta)<\infty$, then
\[\mathcal{T}_E^{\mathrm{rad}}=\Delta_\beta.\]
\item If $\beta<1$ and $\sup_{\delta<1} \delta^\beta N(E,\delta)=\infty$, then
\[ \mathcal{T}_E^{\mathrm{rad}} = \Delta_\beta \backslash [P_{2,\beta}, P_{3,\beta}^{\mathrm{rad}}]. \]
\item If $\beta=1$, then
\[ \mathcal{T}_{E}^\mathrm{rad} = \{ (\tfrac1p,\tfrac1q)\in \Delta_1\,:\, \tfrac1p<\tfrac{d-1}{d}\;\;\text{ or }\;\;
\sup_{\delta\in (0,1)} \delta (\log\tfrac1\delta)^{\frac{q}{d}} N(E,\delta)<\infty \}.
\]
\end{enumerate}
\end{thm}

\begin{remarka}
For $\beta=1$, and the corresponding endpoint case
$p_d=\frac{d}{d-1}$, the operator
$M_E$ maps $L^p_{{\text{\rm rad}}}$ to $L^q$ if and only if
$q\le dp_d$ and \[\sup_{\delta<1} (\log (1/\delta))^{1/d} N(E,\delta)^{1/q} <\infty.\]
Note that the displayed condition is dependent on $q$, in contrast to the case $\beta<1$, where the endpoint bounds for $p=\frac{d-1+\beta}{d-1}$, $q\le p_d$ involve the $q$-independent condition $\sup_{\delta<1} \delta^\beta N(E,\delta) <\infty$.
\end{remarka}

Figure \ref{figureradvsnonrad} relates the result of Theorem \ref{thm:highdim} to the results for $M_E$ acting on general (not necessarily radial) $L^p$ functions, see \cite{AHRS},
\cite{RoosSeeger}.
\begin{figure}[ht]\label{figureradvsnonrad}
\begin{tikzpicture}[scale=4.25]
\draw (0,0) [->] -- (0,1) node [left] {$\frac1q$};
\draw (0,0) [->] -- (1,0) node [below] {$\frac1p$};

\coordinate (Q1) at (0,0);
\coordinate (Q2) at (.71, .71);
\coordinate (Q3) at (.69, .31);
\coordinate (Q4) at (.45, .15);
\coordinate (Q5) at (.685, .225);
\coordinate (Q6) at (.67, 0);

\fill (Q1) node [left] {$P_1$} circle [radius=.25pt];
\fill (Q2) node [above] {$P_2$} circle [radius=.25pt];
\fill (Q3) node [right] {$P_3$} circle [radius=.25pt];
\fill (Q4) node [below] {$P_4$} circle [radius=.25pt];
\fill (Q5) node [below] {$P_3^{{\text{\rm rad}}}$} circle [radius=.25pt];

\draw[dashed,opacity=.3] (1,0) -- (0,1);
\draw[dashed,opacity=.3] (0,0) -- (1,.33);
\draw[dashed,opacity=.3] (0,0) -- (1,1);
\draw[dashed,opacity=.3] (Q3) -- (Q5);
\draw[dashed,opacity=.3] (Q5) -- (Q6);

\fill[opacity=.2] (Q1) -- (Q2) -- (Q3) -- (Q4) -- cycle;
\fill[opacity=.1] (Q4) -- (Q5) -- (Q3);
\draw[opacity=.6] (Q1) -- (Q2) -- (Q3) -- (Q4) -- cycle;
\end{tikzpicture}
\caption{The triangle $\Delta_\beta$.
}
\end{figure}
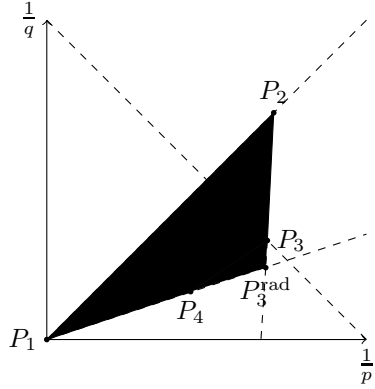
For parameters $0\le \beta\le \gamma\le 1$ let $\mathcal{Q}(\beta,\gamma)$ denote the closed quadrangle with vertices
\begin{equation*}
\begin{gathered}
P_1=(0,0), \qquad \qquad P_{2,\beta} =(\tfrac{d-1}{d-1+\beta}, \tfrac{d-1}{d-1+\beta}),
\\
P_{3,\beta}=(\tfrac{d-\beta}{d-\beta+1}, \tfrac{1}{d-\beta+1}), \qquad P_{4,\gamma}=(\tfrac{d(d-1)}{d^2+2\gamma-1},\tfrac{d-1}{d^2+2\gamma-1}).
\end{gathered}
\end{equation*}
If $\gamma=\dim_{\mathrm{qA}}E$ is the quasi-Assouad dimension of $E$ (for definitions see below) then
$\mathcal{Q}(\beta,\gamma) \subset \overline{\mathcal{T}_E}\subset \mathcal{Q}(\beta,\beta)$; moreover it was shown in \cite{RoosSeeger} that a closed convex set $\mathcal{W}\subset [0,1]^2$ takes the form $\mathcal{W}=\overline{\mathcal{T}_E}$ for some $E$ with ${\dim}_{\mathrm M}E=\beta$, ${\dim}_{\mathrm{qA}}(E)={\gamma}$ if and only if
\[ \mathcal{Q}(\beta,\gamma) \subset{{\mathcal{W}}}\subset \mathcal{Q}(\beta,\beta).\]
We note that the radial type set is strictly larger (here for $d\ge 3$, $\beta<1$): ${{\mathcal{Q}}}(\beta,\beta)\subsetneqq \overline{{{\mathcal{T}}}_E^{{\text{\rm rad}}}}$. This is expected for $q>p$, as the non-radial Knapp examples no longer apply in the radial case.

\subsubsection*{The case $d=2$} We shall now describe our results for $L^p_{{\text{\rm rad}}}\to L^q$ boundedness of $M_E$ when $d=2$.
In order to formulate these we need to first recall some definitions from fractal geometry beyond Minkowski dimension.
The \emph{Assouad dimension}
$\mathrm{dim}_\mathrm{A}E$ is defined as the infimum of all $a>0$ such that there exists $c\in (0,\infty)$ such that for all $\delta\in (0,1)$ and all intervals $J\subset [1,2]$ with $|J|\ge \delta$,
\begin{equation}\label{eqn:NAssouad}
N(E\cap J,\delta)\le c\,\Big(\tfrac{|J|}{\delta}\Big)^a.
\end{equation}

The {\em upper Assouad spectrum} of $E$ is the function $\theta\mapsto \overline{\mathrm{dim}_{\mathrm{A},\theta}}\,E$
defined for every $\theta\in [0,1)$ as the infimum of all $a>0$ such that there exists $c>0$ such that \eqref{eqn:NAssouad} holds
for all $\delta\in (0,1)$ and all intervals $J\subset [1,2]$ of length $|J|\ge \delta^\theta$. The upper Assouad spectrum (Fraser--Hare--Hare--Troscheit--Yu \cite{FHHTY19}) is a variant of the Assouad spectrum (Fraser--Yu \cite{FY18a}, \cite{FY18b}), which arises when $|J|\ge \delta^\theta$ is replaced by $|J|=\delta^\theta$ and is denoted $\dim_{\mathrm{A},\theta}E$. We refer to Jonathan Fraser's book \cite{FraserBook} for an introduction to Assouad-type dimensions. For $\theta=0$ we recover
${\mathrm{dim}}_{\mathrm{M}}\,E=\overline{\mathrm{dim}_{\mathrm{A},0}}\,E$.
The upper Assouad spectrum extends to a continuous function on $[0,1]$ and the limit
\begin{equation*}
\mathrm{dim}_{\mathrm{qA}}E=\lim_{\theta\to 1-} \overline{\mathrm{dim}_{\mathrm{A},\theta}}E
\end{equation*}
is called the {\em quasi-Assouad dimension} (L\"u--Xi \cite{LX16}). From the definitions one sees that always $\dim_\mathrm{qA} E\le \dim_\mathrm{A} E$ and in some examples the inequality is strict; e.g. for $E=\{1+2^{-\sqrt n}: n\in {{\mathbb{N}}}\}$ we have $\dim_{\mathrm{qA}}\!E =0$ and $\dim_{\mathrm A}\!E=1$.

Rutar \cite{Rutar24} gave a simple characterization of the functions that can occur as the (upper) Assouad spectrum of a set.
We shall adopt the convention that when a set $E\subset [1,2]$ is given, we denote
by $\gamma$ either the quasi-Assouad or the Assouad dimension, depending on the context.
Then $0\le \beta\le\gamma \le 1$.

To state our result for $d=2$ we introduce, for $\alpha\in\mathbb{R}$, the quantity
\begin{equation}\label{eq:nudagger}
\nu^\sharp(\alpha) := \limsup_{\delta\to 0} \frac{\log \sup_{\delta\le |J|\le 1} |J|^{-\alpha} N(E\cap J, \delta)}{\log(\frac1\delta)} \,.
\end{equation}
\begin{thm}\label{thm:maintwodim}
Let $d=2$ and $E\subset [1,2]$, and $\beta=\dim_{\mathrm M} E$. Then
\[ \overline{\mathcal{T}_E^{\mathrm{rad}}} = \Delta_\beta \, \cap \, \big\{ (\tfrac1p,\tfrac1q): \tfrac1q \nu^\sharp(\tfrac{q}2-1)+\tfrac1p - \tfrac1q \le \tfrac 12 \big\}. \]

\end{thm}

Observe $\nu^\sharp(\alpha)=\beta$ if $\alpha\le 0$.
In practice, it may be difficult to compute $\nu^\sharp(\alpha)$ for $\alpha>0$. However, the definition implies for $\alpha\ge 0$ (see Lemma \ref{lem:nusharp}),
\begin{equation}\label{eqn:fracprop}
\max(\alpha,\beta)\le \nu^\sharp(\alpha)\le \max(\alpha, (1-\tfrac{\beta}{\gamma})\alpha+\beta),
\end{equation}
where $\gamma=\dim_{\mathrm{qA}}E$.
In particular,
$\nu^\sharp(\alpha)=\alpha \text{ if $\alpha\ge \gamma$.}$
\begin{remarka} To see how $\nu^\sharp$ relates to the Assouad spectrum, suppose that
\begin{equation*}
\sup_{|J|=\delta^\theta} N(E\cap J,\delta)\approx \delta^{\nu(\theta)}
\end{equation*}
for a function $\nu$ and for all $\delta\in (0,1), \theta\in [0,1]$ with constants independent of $\delta,\theta$.
In this case, the Assouad spectrum is given by
$\dim_\mathrm{A,\theta} E=-\nu(\theta)/(1-\theta)$
for $\theta\in [0,1)$ (and if the Assouad spectrum is non-decreasing, this also equals the upper Assouad spectrum $\overline\dim_{\mathrm{A},\theta}E$).
Then $\nu^\sharp$ is equal to the Legendre transform of the function $\nu$, i.e. the right-hand side of \footnote{Added in the revised version: The identity \eqref{eqn:nusharplegendre} continues to hold for all sets $E\subset [1,2]$, where $\nu(\theta)=-(1-\theta)\dim_{\mathrm{A},\theta}E$ (see \cite[Thm.1.2]{BeltranRoosRutarSeeger}).}
\begin{equation}\label{eqn:nusharplegendre}
\nu^\sharp(\alpha) = \sup_{\theta\in [0,1]} \alpha \theta - \nu(\theta).
\end{equation}
Note that $\nu$ is not necessarily convex.
\end{remarka}

In two dimensions, the triangle $\Delta_{\beta}$ has vertices
\begin{equation*}
P_1= (0,0), \quad P_{2,\beta} =(\tfrac{1}{1+\beta}, \tfrac{1}{1+\beta}), \quad P_{3,\beta}^{\mathrm{rad}}=(\tfrac{2}{3+\beta}, \tfrac{1}{3+\beta}).
\end{equation*}
For $2\gamma-\beta>1$ we denote by ${{\mathcal{Q}}}^{\mathrm{rad}}_{\beta,\gamma}$ the closed quadrangle with vertices
\[ P_1,\, P_{2,\beta},\, P_{4,\gamma}^{\mathrm{rad}}=(\tfrac1{1+\gamma},\tfrac1{2(1+\gamma)}),\]
\[P_{5,\beta,\gamma}^{\mathrm{rad}} = (\tfrac{(1-\beta)(2-\beta/\gamma) + 2(1-\beta/\gamma)}{2( (1-\beta) + 2(1-\beta/\gamma) )},\tfrac{1-\beta/\gamma}{1-\beta + 2(1-\beta/\gamma)}). \]
Then Theorem \ref{thm:maintwodim} and \eqref{eqn:fracprop} imply the following (also see Figure \ref{fig:d=2}).
\begin{cor}\label{cor:d=2}
Let $d=2$ and $E\subset [1,2]$, and $\beta=\dim_{\mathrm M} E$, $\gamma=\dim_{\mathrm{qA}} E$.
\begin{enumerate}[(i)]
\item If $2\gamma-\beta\le 1$, then
\[\overline{\mathcal{T}_E^{\mathrm{rad}}}=\Delta_\beta.\]
\item If $2\gamma-\beta>1$, then
\[ \mathcal{Q}^{\mathrm{rad}}_{\beta,\gamma}\subset \overline{\mathcal{T}_E^{\mathrm{rad}}}\subset \Delta_\beta. \]
Moreover, the left inclusion is sharp in the sense that there exist sets $E$ for which the inclusion is an equality.
\end{enumerate}
\end{cor}

\begin{figure}
\begin{tikzpicture}[scale=8]
\draw (0,0) [->] -- (0,1) node [left] {$\frac1q$};
\draw (0,0) [->] -- (1,0) node [below] {$\frac1p$};

\coordinate (P1) at (0,0);
\coordinate (P2) at (.667, .667);
\coordinate (P3rad) at (.571, .286);
\coordinate (P4rad) at (.5, .25);
\coordinate (P5rad) at (.583, .333);

\def\ptsize{.15pt}

\fill (P1) node [left] {$P_1$} circle [radius=\ptsize];
\fill (P2) node [above] {$P_2$} circle [radius=\ptsize];
\fill (.571, .286) circle [radius=\ptsize];
\fill (.63, .275) node {$P_3^\mathrm{rad}$} circle [radius=0pt];
\fill (P4rad) node [below] {$P_4^\mathrm{rad}$} circle [radius=\ptsize];
\fill (P5rad) node [right] {$P_5^\mathrm{rad}$} circle [radius=\ptsize];

\draw[dashed,opacity=.3] (0,.25) -- (1,.25);
\draw[dashed,opacity=.3] (0,0) -- (1,1);
\draw[dashed,opacity=.3] (0,1) -- (1,0);
\draw[dashed,opacity=.3] (0,0) -- (1,.5);

\draw[opacity=.6] (P1) -- (P2) -- (P5rad)-- (P4rad) -- cycle;
\fill[opacity=.2] (P1) -- (P2) -- (P5rad)-- (P4rad) -- cycle;

\draw[opacity=.5] (P1) -- (P2) -- (P3rad) -- cycle;
\fill[opacity=.1] (P1) -- (P2) -- (P3rad) -- cycle;
\end{tikzpicture}

\caption{If $d=2, \beta=0.5, \gamma=1$, then $\overline{\mathcal{T}^\mathrm{rad}_E}$ is contained in $\Delta_\beta$ and contains $\mathcal{Q}_{\beta,\gamma}^\mathrm{rad}$. The condition $2\gamma-\beta>1$ says that the point $P_3^\mathrm{rad}=P_{3,\beta}^\mathrm{rad}$ lies above the horizontal line $q=2(1+\gamma)$.
}
\label{fig:d=2}
\end{figure}
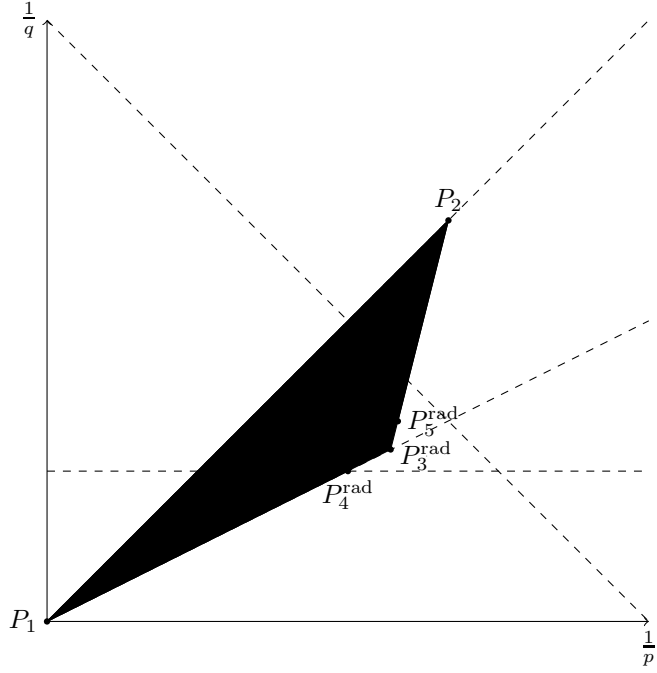

Regarding the boundary of ${{\mathcal{T}}}_E^{\mathrm{rad}}$, it was shown in \cite{SeegerWaingerWright1997} that it includes the segment $[P_1,P_{2,\beta})$. In the $L^p_{\mathrm{rad}} \to L^q$ category we can ensure the following endpoint results, which feature the Assouad dimension instead of the quasi-Assouad dimension.

\begin{thm}\label{thm:2dendpoint}
Let $d=2$, $E \subset [1,2]$, $\beta=\dim_{\mathrm M}E$ and $\gamma= {\dim}_{\mathrm{A}}E.$ Then the following hold.

\begin{enumerate}[(i)]
\item \label{item:2b-g<1}
For $2\gamma-\beta<1$, \[ {{\mathcal{T}}}_E^{\mathrm{rad}}=\Delta_\beta \,\iff \sup_{\delta<1} \delta^\beta N(E,\delta)< \infty.\]

\item
For $2{\gamma}-\beta=1$, $\beta<1$ and if $\sup_{0<\delta <1/2} \delta^\beta N(E,\delta)< \infty$,
\[\Delta_\beta\setminus \{P_{3,\beta}^{{\text{\rm rad}}} \} \subset {{\mathcal{T}}}_E^{\mathrm{rad}} \subset\Delta_\beta \]
\item
For $2\gamma-\beta=1 $ and if $\sup_{\delta<1} \delta^\beta N(E,\delta)= \infty$, then \[{{\mathcal{T}}}_E^{\mathrm{rad}}= \Delta_\beta \backslash [P_{2,\beta},P_{3,\beta}^{\mathrm{rad}}].\]
\item For $2\gamma-\beta > 1$, then $[P_1, P_{4,\gamma}^{\mathrm{rad}}) \subset {{\mathcal{T}}}_E^{\mathrm{rad}}$.
\item If $\beta=1$ and if $\sup_{\delta<1} \delta\log(\frac 1\delta) N(E,\delta)=\infty$ then $M_E:L^p_{{\text{\rm rad}}}\to L^q$ is bounded if and only if $p>2$ and $p\le q\le 2p.$
\end{enumerate}
\end{thm}
For the case $\beta=\gamma<1$, part \eqref{item:2b-g<1} gives a characterization of the radial type set; this applies to self-similar set such as the Cantor middle third set. Note that the last part applies to the full set $E=[1,2]$, in which case $M_E:L^p_{{\text{\rm rad}}}\to L^q$ is bounded if and only if $p>2$, $p\le q\le 2p.$

\begin{remarka} For $2{\gamma}-\beta>1$, $\beta<1$, $L^p$ boundedness holds for a part of $[P_{2,\beta},P_{5,\beta, \gamma} ^{{\text{\rm rad}}})$ if $\sup_{\delta<1} \delta^\beta N(E,\delta)<\infty.$ However, it is currently an open question whether $L^p\to L^q$ boundedness holds for all $(\frac 1p, \frac 1q)\in [P_{2,\beta}, P_{5,\beta,\gamma}^{{\text{\rm rad}}}).$
\end{remarka}

\subsubsection*{Notation}

Given a list of objects $L$ and real numbers $A$, $B \geq 0$, here and throughout we write $A \lesssim_L B$ or $B \gtrsim_L A$ to indicate $A \leq C_L B$ for some constant $C_L$ which depends only on items in the list $L$. We write $A \sim_L B$ to indicate $A \lesssim_L B$ and $B \lesssim_L A$.

\subsubsection*{Structure of the paper}
\begin{itemize}
\item \S \ref{sec:nusharp} discusses properties of the dimensional quantity $\nu^\sharp$, in particular proving \eqref{eqn:fracprop} and Corollary \ref{cor:d=2} assuming Theorem \ref{thm:maintwodim}.
\item \S\ref{sec:nec} concerns necessary conditions for $L^p_{\mathrm{rad}} \to L^q$ boundedness of $M_E$.
\item In \S\ref{sec:higher} the upper bounds for $d\ge 3$ in Theorem \ref{thm:highdim} are proved.
\item In \S \ref{sec:2d}, the upper bounds for $d=2$ and Theorems \ref{thm:maintwodim} and \ref{thm:2dendpoint} are proved.
\end{itemize}

\subsubsection*{Acknowledgements}

This research was supported through the program {Oberwolfach Research Fellows} by Mathematisches Forschungsinstitut Oberwolfach in 2023. The authors were supported in part by the grants RYC2020-029151-I and PID2022-140977NA-I00, funded by MICIU/AEI/10.13039/501100011033, by ``ESF Investing in your future" and by FEDER, UE (D.B.) and by the National Science Foundation grants
DMS-2154835 (J.R.), DMS-2348797 (A.S.). J.R. also thanks the Hausdorff Research Institute for Mathematics in Bonn for providing a pleasant working environment during the Fall 2024 trimester program.
We thank Alex Rutar for helpful comments on a previous version of this paper.

\subsubsection*{Remark} Shuijiang Zhao \cite{ZhaoSh24} has independently obtained results for spherical maximal operators on radial functions that overlap with ours. His results in two dimensions are essentially sharp in the special case of quasi-Assouad regular sets (or finite unions of those).

\section{Properties of \texorpdfstring{$\nu^\sharp$}{ dimensional quantities}}\label{sec:nusharp}
Let $E\subset [1,2]$ be nonempty, $\alpha\in\mathbb{R}$ and $\nu^\sharp(\alpha)$ as in \eqref{eq:nudagger}.
Equivalently, $\nu^\sharp(\alpha)$ equals the infimum over all $\nu\ge 0$ such that for every $\varepsilon>0$ there exists $C_\varepsilon\in (0,\infty)$ such that for all intervals $J\subset [1,2]$ with $|J|\ge \delta$,
\[ N(E\cap J,\delta) \le C_\varepsilon |J|^\alpha \delta^{-\nu-\varepsilon}. \]
If $\alpha\le 0$, then $\nu^\sharp(\alpha)=\beta$. If $\alpha>0$, then $\nu^\sharp(\alpha)$ may be difficult to compute, but we have the following bounds.
\begin{lemma}\label{lem:nusharp}
Let $\gamma=\mathrm{dim}_\mathrm{qA}\,E>0$. For all $\alpha\ge 0$,
\[\max(\alpha,\beta)\le\nu^\sharp(\alpha)\le \max(\alpha, (1-\tfrac{\beta}{\gamma})\alpha+\beta).\]
In particular, $\nu^\sharp(\alpha)=\alpha$ if $\gamma\ge \alpha$.
\end{lemma}

\begin{proof}
Choosing an interval $J$ of length $|J|=\delta$ shows that $\nu^\sharp(\alpha)\ge \alpha$ and choosing $J=[1,2]$ shows $\nu^\sharp(\alpha)\ge \beta$.

To prove the upper bound we use three different estimates for $N(E\cap J,\delta)$ depending on the magnitude of $|J|$.
For the remainder of this proof set $\theta=\theta(J)\in [0,1]$ such that
$|J|=\delta^\theta.$
We note the following:
\begin{itemize}
\item[-] For $\theta$ near $0$ we can estimate $N(E\cap J,\delta)\le N(E,\delta)$. Thus, by the definition of the upper Minkowski dimension $\beta$, for every
$\varepsilon>0$ there exists $C'_\varepsilon\in (0,\infty)$ such that
\begin{equation}\label{eqn:nusharppf1}
N(E\cap J,\delta)\le C'_\varepsilon \delta^{-(\beta+\varepsilon)}
\end{equation}
for all $J\subset [1,2]$.
\medskip
\item[-] For $\theta$ away from $0$ and away from $1$, we will use the definition of the quasi-Assouad dimension $\gamma$: for every $\varepsilon>0$ there exists $C_{\varepsilon}\in (0,\infty)$ such that
\begin{equation}\label{eqn:nusharppf2}
N(E\cap J,\delta) \le C_{\varepsilon} \delta^{-(1-\theta)(\gamma+\varepsilon)}
\end{equation}
for all $J\subset [1,2]$ with $\theta\le 1-\varepsilon$.
\medskip
\item[-] For $\theta$ close to $1$ we
estimate $N(E\cap J,\delta)\le N(J,\delta)$, which gives
\begin{equation}\label{eqn:nusharppf3}
N(E\cap J,\delta) \le \delta^{-(1-\theta)} +1
\end{equation}
for all $J\subset [1,2]$.
\end{itemize}

With this in mind, consider two cases.
\medskip

{\em Case 1:} $\alpha\le \gamma$. Then we need to show $\nu^\sharp(\alpha)\le (1-\frac\beta\gamma)\alpha+\beta$.
Note that the exponents on the right-hand sides of \eqref{eqn:nusharppf1} and \eqref{eqn:nusharppf2} coincide up to $\varepsilon$ when $\theta=1-\frac\beta\gamma$. This motivates the estimate
\[ N(E\cap J,\delta) |J|^{-\alpha}\le
\begin{cases}
C'_\varepsilon \delta^{-(\alpha\theta + \beta + \varepsilon)} &\text{ if } \theta\le 1-\frac\beta\gamma
\\
C_\varepsilon \delta^{-(-\theta(\gamma-\alpha) + \gamma+\varepsilon)} &\text{ if } 1-\frac\beta\gamma< \theta\le 1-\varepsilon
\\
\delta^{-(\alpha\theta+\varepsilon)}+1 &\text{ if } 1-{\varepsilon} \le \theta\le 1.
\end{cases}
\]
In the first two cases the largest value is assumed when $\theta=1-\frac\beta\gamma$, and in the third case the largest value is assumed when $\theta=1$.
Hence
\[ N(E\cap J,\delta)|J|^{-\alpha} \le \max(C_\varepsilon,C'_\varepsilon) \delta^{-(\alpha(1-\frac\beta\gamma)+\beta+\varepsilon)} + \delta^{-(\alpha+\varepsilon)} +1. \]
Since $\alpha\le (1-\frac\beta\gamma)\alpha+\beta$, the claim follows.

{\em Case 2:} $\alpha\ge \gamma$.
Then $\alpha\ge (1-\frac\beta\gamma)\alpha+\beta,$ so we need to show $\nu^\sharp(\alpha)\le \alpha$.
By \eqref{eqn:nusharppf2} and \eqref{eqn:nusharppf3},
\[ N(E\cap J,\delta) |J|^{-\alpha} \le C_{\varepsilon} \delta^{-(\theta(\alpha-\gamma) + \gamma+\varepsilon)} {{\mathbbm 1}}_{\{\theta\le 1-\varepsilon\}} + ( \delta^{-(\alpha\theta + \varepsilon)}+1) {{\mathbbm 1}}_{\{1\ge \theta>1-\epsilon\}} \]
Since $\alpha-\gamma\ge 0$, the first term on the right hand side is $\le C_\varepsilon \delta^{-((1-\varepsilon)(\alpha-\gamma)+\gamma+\varepsilon)},$
which is $\le C_\varepsilon \delta^{-(\alpha+\varepsilon')}$
with $\varepsilon'=\varepsilon (1-(\alpha-\gamma)),$
and the second term is $\le \delta^{-(\alpha+\varepsilon)}$.
By definition of $\nu^\sharp(\alpha)$, this implies $\nu^\sharp(\alpha)\le \alpha$.
\end{proof}

We now discuss how to deduce Corollary \ref{cor:d=2} from Theorem \ref{thm:maintwodim}. We note that this implication is unrelated to spherical averaging operators.

\begin{proof}[Proof of Corollary \ref{cor:d=2}, given Theorem \ref{thm:maintwodim}]

By Theorem \ref{thm:maintwodim}, $\overline{{{\mathcal{T}}}_E^{\mathrm{rad}}} \subset \Delta_\beta$.
For the other inclusion, again by Theorem \ref{thm:maintwodim}, note that $(\frac1p,\frac1q)\in\overline{\mathcal{T}_E^{\mathrm{rad}}}$ if and only if
\begin{enumerate}[(a)]
\item $\frac{1}{q} \leq \frac{1}{p}$,
\item $\frac{2}{p} + \frac{\beta-1}{q} \leq 1$, and
\item $\frac{1}{q}\nu^{\sharp}(\frac{q}{2}-1) + \frac{1}{p} - \frac{1}{q} \leq \frac{1}{2}$.
\end{enumerate}
Since $\nu^{\sharp}(\alpha) \geq \alpha$, the scaling condition $\frac{1}{q}\geq \frac{1}{2p}$ follows from (c) for $q \geq 2$ (and the condition also holds for $q\le 2$ since $p\ge 1$).
\medskip

(i) Assume that $2\gamma - \beta \leq 1$. From \eqref{eqn:fracprop} we have $\nu^\sharp(\frac{q}{2}-1)=\frac{q}{2}-1$ if $\frac{1}{q} \leq \frac{1}{2(1+\gamma)}$. Note $P_{3,\beta}^{\mathrm{rad}}=(\tfrac{2}{3+\beta}, \frac{1}{3+\beta})$ and $\frac{1}{3+\beta} \leq \frac{1}{2(1+\gamma)}$ for $2\gamma - \beta \leq 1$, therefore $P_{3,\beta}^{\mathrm{rad}} \in \overline{{{\mathcal{T}}}_E^{\mathrm{rad}}}$. Hence $\overline{{{\mathcal{T}}}_E^{\mathrm{rad}}} = \Delta_\beta$ by convexity, since $P_1, P_{2,\beta} \in \overline{{{\mathcal{T}}}_E^{{\text{\rm rad}}}}$.

\medskip

(ii) Assume that $2\gamma-\beta > 1$. In this case $\frac{1}{3+\beta} > \frac{1}{2(1+\gamma)}$. If $\frac{1}{q} \leq \frac{1}{2(1+\gamma)}$ we have $\nu^\sharp(\frac{q}{2}-1)=\frac{q}{2}-1$ and condition (c) yields the point $P_{4,\gamma}^{\mathrm{rad}}$. If
$\frac{1}{q}> \frac{1}{2(1+\gamma)}$, we have $\nu^{\sharp}(\frac{q}{2}-1) \leq (1-\frac{\beta}{\gamma})(\frac{q}{2}-1) +\beta$, and in particular (c) is satisfied if $
(1-\frac{\beta}{\gamma})(\frac{1}{2}-\frac{1}{q}) + \frac{\beta}{q} + \frac{1}{p}-\frac{1}{q} \leq \frac{1}{2}.$
A computation shows that this matches condition (b) at the point $P_{5,\beta,\gamma}^{\mathrm{rad}}$. Therefore, ${{\mathcal{Q}}}_{\beta,\gamma}^{\mathrm{rad}}\subset \overline{{{\mathcal{T}}}_E^{\mathrm{rad}}}$.

It remains to show sharpness, in the sense that the inclusion $\mathcal{Q}^{\mathrm{rad}}_{\beta,\gamma}\subset \overline{\mathcal{T}_E^{\mathrm{rad}}}$ may be an equality for some sets $E$. This happens when the upper bound in \eqref{eqn:fracprop} is an equality.
For example, if $E$ is the $(\beta,\gamma)$ Assouad-regular set constructed in \cite[\S 6.2]{RoosSeeger} and $\nu(\theta)=-\min((1-\theta)\gamma,\beta)$,
then
\begin{equation*}
\sup_{|J|=\delta^\theta} N(E\cap J,\delta)\approx \delta^{\nu(\theta)}
\end{equation*}
uniformly in $\delta,\theta$ (in particular, the Assouad spectrum and upper Assouad spectrum both equal $\gamma(\theta)=\min(\frac{\beta}{1-\theta},\gamma)$).
This implies \[\nu^\sharp(\alpha)=\max(\alpha, (1-\tfrac\beta\gamma)\alpha+\beta)\] as required.

{\em Remark.}
For $\gamma=1$ and $\beta\in (0,1)$, another example is given by the convex sequences $E=\{1+n^{1-\beta^{-1}}\,:\,n\ge 1\}$ which also have (upper) Assouad spectrum equal to $\gamma(\theta)=\min(\frac{\beta}{1-\theta},1)$.
Additional examples are provided by the Moran set constructions in \cite[\S 3]{Rutar24} (also see \cite{BR22}).
\end{proof}

\section{Necessary conditions}\label{sec:nec}

We start with the necessary condition $p \leq q$. We provide a direct proof, since for the class of radial functions one cannot directly appeal to the standard example of Hörmander \cite{hormander1960} for translation-invariant operators.

\begin{lemma}\label{lem: pq nec}
Let $1 \leq p, q < \infty$. Assume that $\| M_E \|_{L^p_{\mathrm{rad}} \to L^q} \lesssim 1$. Then $p \leq q$.
\end{lemma}

\begin{proof}
Let $k >0$ be a large parameter. For $n >0$, let \[ I^{k,*}_n:=[2^k+8n+1, 2^k+ 8n+7], \qquad I^k_n:=[2^k+8n+3, 2^k+ 8n+5]\] and define
\begin{equation*}
f_k(x):= 2^{-k(d-1)/p} \sum_{n=1}^{2^{k-5} }{{\mathbbm 1}}_{I^{k,*}_n} (|x|).
\end{equation*}
Then $\|f_k\|_p{\lesssim} 2^{k/p}.$ Let $t\in E \subset [1,2]$ and observe that
\begin{align*}
A_{t} [ {{\mathbbm 1}}_{I^{k,*}_n}(|\cdot|)] (x) &=1, \quad \text{ for } |x|\in I_{n}^k \\
A_{t} [ {{\mathbbm 1}}_{I^{k,*}_n}(|\cdot|)] (x) &=0, \quad \text{ for } |x|\in I_{n'}^k, \quad n \neq n'.
\end{align*}
Therefore
\begin{align*}
\|A_t f_k\|_q&\ge \Big(\sum_{n=1}^{2^{k-5}} 2^{-k(d-1)q/p} \int_{\{x\,:\, |x| \in I_{n}^k\}} \Big|\sum_{n'=1}^{2^{k-5}} A_t [{{\mathbbm 1}}_{I_{n'}^{k,*}} (|\cdot|) ](x)\Big|^q {{\text{\,\rm d}}} x\Big)^{1/q} \\
& = \Big(\sum_{n=1}^{2^{k-5}} 2^{-k(d-1)q/p} \int_{I_{n}^k} r^{d-1} {{\text{\,\rm d}}} r\Big)^{1/q} \gtrsim 2^{-k(d-1) (1/p-1/q) } 2^{k/q}.
\end{align*}
Consequently,
\begin{equation}\label{eq:ratio pq}
\frac{\|M_E f_k\|_q}{\| f_k\|_p} \gtrsim 2^{-kd(1/p-1/q)}
\end{equation}
and since $\| M_E \|_{L^p_{\mathrm{rad}} \to L^q} \lesssim 1$ by assumption, we must have $p \leq q$ by letting $k \to \infty$ in \eqref{eq:ratio pq}.
\end{proof}

The next lemma is standard, and can be found, for instance, in \cite{AHRS}. We note that the first condition is the usual scaling condition from fixed-time averages and thus does not depend on the set of dilations $E$.

\begin{lemma} \label{lem: nec easy}
Let $1 \leq p \leq q < \infty$, $d \geq 2$.
\begin{enumerate}[(i)]
\item If $\| M_{[1,2]} \|_{L^p_{\mathrm{rad}} \to L^q} \lesssim 1$, then $q \leq pd$.
\item For any $E \subset [1,2]$
\begin{equation}\label{eq:nec easy}
\| M_E \|_{L^p_{\mathrm{rad}} \to L^q} \gtrsim \sup_{\delta<1 } N(E, \delta)^{1/q} \delta^{d-1+\frac{1}{q}-\frac{d}{p}}.
\end{equation}
\end{enumerate}
\end{lemma}

The previous conditions imply that $L^p_{\mathrm{rad}} \to L^q$ bounds for $M_E$ cannot hold for $(\frac{1}{p}, \frac{1}{q}) \not \in \Delta_\beta$. This was essentially implicit in previous works such as \cite{AHRS}.

\begin{cor}\label{cor: Delta beta nec}
Let $d \geq 2$, $E \subset [1,2]$ with $\dim_M E = \beta$. Then ${{\mathcal{T}}}_E^{\mathrm{rad}} \subset \Delta_\beta$.
\end{cor}

\begin{proof}
The triangle $\Delta_\beta$ is determined by the intersection of the lines $\frac{1}{p}=\frac{1}{q}$, $\frac{1}{q}=\frac{1}{pd}$ and $d-1 + \frac{1-\beta}{q}=\frac{d}{p}$. Using the definition of Minkowski dimension, these correspond to the necessary conditions in Lemma \ref{lem: pq nec} and Lemma \ref{lem: nec easy}, respectively.
\end{proof}

We note that when $p=\frac{d}{d-1}$ there is a more refined necessary condition than \eqref{eq:nec easy}.

\begin{lemma}\label{lem:lowerbd-pd}
Let $p_d= \frac{d}{d-1}$. Then
\[ \| M_E\|_{L^{p_d} _{{\text{\rm rad}}}\to L^q} {\gtrsim} \sup_{\delta<1} N(E,\delta)^{1/q} \delta^{1/q} [\log (1/\delta)]^{1/d}. \]
\end{lemma}
\begin{proof} First consider $f(x)=|x|^{1-d} \log (|x|^{-1})^{-1} {{\mathbbm 1}}_{[0,\frac 12]}(|x|)$, the familiar example by Stein \cite{SteinPNAS1976}. Then $f\in L^{p_d}$ and $M_E f(x)=\infty$ for $|x|\in\overline E$. Hence if $\|M_E\|_{L^{p_d}_{{\text{\rm rad}}} \to L^q} <\infty$ then $\overline E$ must be of measure zero.

Let $\delta\le 10^{-2}$ and define
\[ g(x):=|x|^{1-d} \big( \log (\tfrac{1}{|x|} )\big)^{\frac{1-d}{d}} {{\mathbbm 1}}_{[\delta^{1/2}, \delta^{1/4} ]}(|x|).\]
Then
\begin{align*}
\|g\|_{p_d}
&\sim \Big( \int_{\delta^{1/2}}^{\delta^{1/4} } s^{-1} (\log \tfrac 1s)^{-1} \, {{\text{\,\rm d}}} s\Big)^{1/p_d}
\sim \Big( \int ^{ \tfrac 12 \log(\frac 1\delta) } _{ \tfrac 14\log(\frac 1\delta) }
u^{-1} {{\text{\,\rm d}}} u\Big)^{1/p_d} \sim 1
\end{align*}
uniformly in $\delta$.

Next, let $D_n=\{r: 2^{-n} \le {{\text{\rm dist}}} (r,E) \le 2^{1-n} \}$.
As stated in \cite[Lemma 2.7]{SeegerWaingerWright1997}, we have that for nonnegative $f$ and $f(x)=f_0(|x|)$,
\[ M_E f(x)\ge c \int_{2^{-n+2}}^1 s^{d-2} f_0(s)\, {{\text{\,\rm d}}} s \quad \text{ for $|x|\in D_n$;} \] indeed,
this is a consequence of formula \eqref{eq:Kt-polar} below.
Thus, for $|x|\in D_n$ and $2^{-n} \le \delta$ we get
\begin{align*} M_E g(x)&
{\gtrsim} \int_{\delta^{1/2}}^{\delta^{1/4}} s^{-1} \log(\tfrac 1s)^{\frac{1-d}{d}} {{\text{\,\rm d}}} s
= \int_{\tfrac 14 \log (\frac 1\delta)} ^{ \tfrac 12 \log( \frac{1}{\delta})} u^{1/d} u^{-1} {{\text{\,\rm d}}} u {\gtrsim} [\log (\tfrac 1\delta)] ^{1/d}.
\end{align*} Therefore, letting $W_\delta:=\{r: {{\text{\rm dist}}} (r,E) \le \delta \}$,
\begin{align*}\|M_E f\|_q {\gtrsim} \Big(\sum_{n:2^{-n}\le \delta} |D_n|\, [\log (\tfrac 1\delta)] ^{q/d}\Big)^{1/q} {\gtrsim} |W_\delta|^{1/q} [\log (\tfrac 1\delta)] ^{1/d};
\end{align*}
here we used that $|\overline E|=0$.
Noting that $|W_\delta|\ge \frac 13 N(E,\delta)\delta$, the asserted lower bound follows.
\end{proof}
We continue with a further necessary condition involving the Assouad spectrum. To the best of our knowledge, this condition has not shown up before in the literature. It turns out to be relevant only for $d=2$.

\begin{lemma}\label{lem: nec hard}
Let $1 \leq p \leq q < \infty$. Given $0 < \delta < 1$ and an interval $J$ of length $|J|> \delta$, define
\begin{align*}
{{\mathfrak{C}}}_E(\delta, J) &:= N(E \cap J, \delta)^{\frac 1q} |J|^{-(d-1)(\frac{1}{2}-\frac{1}{q})} \delta^{\frac{d-1}{2} + \frac{1}{q} - \frac{1}{p}}.
\end{align*}
Then
\begin{equation}\label{eq:nechard}
\| M_E \|_{L^p_{\mathrm{rad}} \to L^q} \gtrsim \sup_{\delta<1}\sup_{|J|>\delta} {{\mathfrak{C}}}_E(\delta,J)\,.
\end{equation}
In particular, if $q\ge 2$ and $\nu^\sharp(\alpha)$ is as in \eqref{eq:nudagger} the $L^p_{{\text{\rm rad}}}\to L^q$ boundedness of $M_E$ implies
\[\tfrac{1}{q}\nu^\sharp\big((d-1)(\tfrac q2-1) \big) + \tfrac{1}{p} - \tfrac{1}{q} \leq \tfrac{d-1}{2}.\]
\end{lemma}

\begin{proof}
Let $1 \leq p \leq q \leq 2$. Then the exponent of the factor $|J|^{-(d-1)(\frac{1}{2}-\frac{1}{q})}$ is positive and thus, for any $0 < \delta < 1$,
\begin{equation*}
\sup_{|J| \geq \delta} {{\mathfrak{C}}}_E(\delta, J)= N(E,\delta)^{\frac 1q}\delta^{\frac{d-1}{2}+\frac{1}{q}-\frac{1}{p}} \lesssim N(E,\delta)^{\frac 1q} \delta^{d-1+\frac{1}{q}-\frac{d}{p}}
\end{equation*}
where the inequality follows since $p \leq 2$. Thus, the desired lower bound follows from Lemma \ref{lem: nec easy}, (ii).

We next assume $2 \leq q < \infty$. Let $0< \delta <1$ and $J=[t_L,t_R] \subset [1/2,5/2]$ be an interval with $|J|\geq \delta$. Let $g_\delta(x):={{\mathbbm 1}}_{[t_L-\delta, t_L+\delta]}(|x|)$. Clearly, $\| g_\delta \|_p \lesssim \delta^{1/p}$. For each $t \in E \cap J$ with $|t - t_L|\geq \delta$, let $I_t:=[t-t_L-\delta/10, t-t_L + \delta/10]$. For any $x \in {\mathbb R}^d$ such that $|x| \in I_t$, let $t(x)=t \in E$.
We claim that
\begin{equation}\label{eq:claim}
|A_{t(x)} g_\delta (x)| \gtrsim \Big(\frac{\delta}{|x|} \Big)^{\frac{d-1}{2}}.
\end{equation}
Thus, if $D: = \bigcup_{t \in E \cap J, |t-t_L|\ge \delta} I_t$, we have $M_E g_\delta (x) \gtrsim (\delta/|x|)^{\frac{d-1}{2}}$ whenever $|x| \in D$. Consequently, noting that $r \leq 2|J|$ for $r \in D$, we have for $q \geq 2$
\begin{align*}
\| M_E g_\delta \|_q & \gtrsim \delta^{\frac{d-1}{2}} \Big( \int_D r^{-\frac{(d-1)q}{2} + d-1 } \mathrm{d}r \Big)^{1/q} \\
& \gtrsim \delta^{\frac{d-1}{2}} |J|^{-(d-1)(\frac{1}{2} - \frac{1}{q})} |D|^{1/q}.
\end{align*}
Using that $|D| \gtrsim N(E \cap J, \delta)^{1/q}\delta^{1/q}$ we have
\begin{equation*}
\frac{\| M_E g_\delta \|_q}{\| g_\delta\|_p} \gtrsim N(E \cap J, \delta)^{1/q} |J|^{-(d-1)(\frac{1}{2}-\frac{1}{q})} \delta^{\frac{(d-1)}{2}+\frac{1}{q} - \frac{1}{p}}= {{\mathfrak{C}}}_E(\delta),
\end{equation*}
as desired.

We now prove the claim \eqref{eq:claim}. Note that, by rotation invariance,
we can assume $x=(x_1,0)$ with $x_1\ge 0$ and $x_1-t(x) =-t_L+c_1\delta $, where $|c_1|\le 1/10$.
Let
\begin{equation*}
R_x:=\{ (y_1,y') \in S^{d-1} : y_1 \geq 0, |y'| \leq c_2^{-1} (\delta/x_1)^{1/2} \}
\end{equation*}
with $c_2^2 > \frac{20}{9}$. Note that if $y \in R_x$ we have $(1-y_1) \leq |y'|^2 \leq \frac{1}{c_2^2}\frac{\delta}{|x_1|}$. Now, a computation shows that for $y \in R_x$,
\begin{equation}\label{eq:support for gdelta}
t_L - \tfrac{\delta}{10} \leq |x-t(x)y| \leq t_L+\delta.
\end{equation}
Indeed, on the one hand, we have
\begin{align*}
|x-t(x)y|^2
&= x_1^2 + (t(x))^2 -2 x_1 t(x) y_1 \\
& = (x_1-t(x))^2 + 2t(x)x_1 (1-y_1) \\
& = (t_L-c_1\delta )^2 + 2 t(x) x_1 (1-y_1) \\
& \leq t_L^2 + c_1^2 \delta^2 -2 c_1 t_L \delta + 2 \frac{t(x)}{t_Lc_2^2} t_L\delta \\
& \leq t_L^2+\delta^2+2t_L\delta -(1-c_1^2) \delta^2-2(1+c_1)t_L\delta +4 c_2^{-2} t_L\delta \\&\le (t_L +\delta)^2.
\end{align*}
In the last line we used that $-2(1+c_1)+4c_2^{-2} <2(-\frac{9}{10} +2c_2^{-2} )<0$, and in the previous one, $t \leq 2t_L$.
On the other hand, since $2t(x)x_1(1-y_1)\ge 0$ we have
\[|x-t(x) y|^2= (t_L-c_1\delta)^2+2t(x)x_1(1-y_1)\ge (t_L-\tfrac 1{10}\delta)^2.\]
In view of \eqref{eq:support for gdelta} and the definition of $g_\delta$, we immediately have
\[A_{t(x)} g_\delta(x)\ge \int_{R_x} g_\delta(x-t(x) y) {{\text{\,\rm d}}}\sigma(y) {\gtrsim} (\sqrt{\delta/|x_1|} )^{d-1}, \]
which verifies the claim \eqref{eq:claim}.

Finally, note that by the definition of $\nu^\sharp(\alpha)$ in \eqref{eq:nudagger}, given $\varepsilon>0$, there exists a sequence $(\delta_m)$ with $\lim_{m\to\infty}\delta_m=0$,
such that \[\sup_{|J| \geq \delta_m} |J|^{-\alpha} N(E\cap J,\delta_m) \geq \delta_m^{\varepsilon-\nu^\sharp (\alpha)}.\]

Using this in \eqref{eq:nechard} with $\alpha= (d-1)(\frac q2-1)$ we see that there exist an interval $J$
such that
\[|J|^{((d-1)\frac q 2-1)\frac 1q} \delta_m^{\varepsilon-\nu^\sharp ( (d-1) \frac q2-1)/q} |J|^{-(d-1)(\frac 12-\frac 1q) } \delta_m^{\frac{d-1}{2}+\frac 1q-\frac 1p} {\lesssim} 1 \] for $\delta_m\le |J|$. Letting $m\to \infty$, this implies
$\tfrac{1}{q}\nu^\sharp((d-1)(\tfrac q2-1)) + \tfrac{1}{p} - \tfrac{1}{q} \leq \tfrac{d-1}{2}$ as desired. \end{proof}

\begin{remarka}
The scaling condition $q \leq pd$ in Lemma \ref{lem: nec easy}, (i), is implicit in the previous lemma by simply taking $|J|=\delta$ in the definition of ${{\mathfrak{C}}}_E(\delta, J)$.
\end{remarka}

\section{Estimates for \texorpdfstring{$d\ge 3$}{higher dimensions}}\label{sec:higher}
As in \cite{Leckband, SeegerWaingerWright1997} it is useful to rewrite the spherical averages, when acting on radial functions, as integral transforms on ${{\mathbb{R}}}^+$.
As shown in \cite{Leckband} one has for $f(x)=f_0(|x|)$

\Be
\label{eq:polar}
A_t f(x)= c_d \int_{||x|-t|}^{|x|+t} K_t(|x|,s) f_0(s)\, {{\text{\,\rm d}}} s
\Ee
with \Be\label{eq:Kt-polar}
K_t(r,s)= \Big( \frac{\sqrt{(r+t)^2-s^2} \sqrt{s^2-(r-t)^2} }{(r+t)^2-(r-t)^2} \Big)^{d-3} \frac{s}{(r+t)^2-(r-t)^2}\,.
\Ee

In this section we consider the higher dimensional case and prove Theorem \ref{thm:highdim}.
As stated in \cite{SeegerWaingerWright1997} the formulas
\eqref{eq:polar} and \eqref{eq:Kt-polar} imply expressions for the maximal functions which are easy to handle in dimension $d\ge 3$.

\begin{lemma}[{\cite[Lemma 3.1]{SeegerWaingerWright1997}}]\label{lem:decomposition d3}
Let $d\ge 3$, $1\le p<\infty$ and set $g(s):=f_0(s)s^{(d-1)/p}$.
Then for $r=|x|$,
\begin{equation*}
M_E f(x)\lesssim {{\mathfrak{M}}}_p g(r) +R_1f_0(r) + R_2 f_0(r),
\end{equation*}
where $r=|x|$ and
\begin{align}
{{\mathfrak{M}}}_p g(r)&:=\sup_{\substack{ t\in E\\r/2<t<3r/2}}
r^{1-d}\Big|\int_{|r-t|}^{r+t}s^{\frac{d-1}{p'}-1} g(s) {{\text{\,\rm d}}} s\Big|,
\label{eq:Mp def}\\
R_1 f_0(r)\,&:=\,\sup_{\substack{t\in[1,2]\\ t\le r/2}}\frac 1t \Big|\int_{r-t}^{r+t}f_0(s) {{\text{\,\rm d}}} s\Big|,
\notag \\
R_2 f_0(r)\,&:=\,\sup_{\substack{t\in [1,2]\\t\ge 3r/2}}\frac 1r\Big|\int_{t-r}^{t+r}f_0(s) {{\text{\,\rm d}}} s\Big|. \notag
\end{align}
\end{lemma}

The dependence of the $p$-range on $\beta$ is only used when we estimate ${{\mathfrak{M}}}_p g$.
We have not kept dependence of the set $E$ in the operators $R_1$ and $R_2$ since the operators with supremum over the full interval $[1,2]$ already satisfy satisfactory estimates. The operator $R_1$ is rather straightforward.

\begin{proposition}\label{Ronelem}
For all $1\le p\le q\le\infty$ we have
\begin{equation*}
\|R_1 f_0\|_{L^q(r^{d-1} \mathrm{d}r)}{\lesssim}
\| f_0\|_{L^p(s^{d-1} \mathrm{d}s)}.
\end{equation*}
\end{proposition}
\begin{proof}
The estimate is trivial for $q=p=\infty$, so we assume that $1 \leq p < \infty$. For any $t \in [1,2]$ and $r \geq 2t$, we have
\begin{align*}
\int_{r-t}^{r+t} |f_0(s)| {{\text{\,\rm d}}} s \lesssim r^{-\frac{d-1}{p}} \int_{r-t}^{r+t} |f_0(s)|\, s^{\frac{d-1}{p}} {{\text{\,\rm d}}} s {\lesssim} r^{-\frac{d-1}p} \int_{r-2}^{r+2} |f_0(s)|\, s^{\frac{d-1}{p}} {{\text{\,\rm d}}} s.
\end{align*}
Thus, by Hölder's inequality
\begin{align*}
R_1f_0(r) \lesssim r^{-\frac{d-1}p} \Big(\int_{r-2}^{r+2} |f_0(s)|^ps^{d-1} {{\text{\,\rm d}}} s\Big)^{1/p}.
\end{align*}
The case $q=\infty$ is immediate. For $1 \leq q < \infty$, we note that $R_1f_0(r)=0$ if $r<2$, and thus $(\int_0^\infty |R_1f_0(r)|^qr^{d-1} {{\text{\,\rm d}}} r)^{1/q}$ is bounded by a constant times
\begin{align*}
\Big(\int_2^\infty r^{-(d-1)(\frac 1p-\frac 1q)q}
\Big(\int_{r-2}^{r+2} |f_0(s)|^ps^{d-1} {{\text{\,\rm d}}} s\Big)^{q/p}{{\text{\,\rm d}}} r\Big)^{1/q}.
\end{align*}
Since $q \geq p$, a standard spatial orthogonality argument implies that the above is further bounded by a constant times $\Big(\int_0^\infty |f_0(s)|^ps^{d-1} {{\text{\,\rm d}}} s\Big)^{1/p}$, concluding the proof.
\end{proof}

The condition $q\le pd$ is necessary for the estimation of $R_2$. This is nontrivial only in the endpoint case $q=pd$.

\begin{proposition}\label{lem:Rtwo}
Let $d \geq 3$. For $1<p<\infty$, $q\le pd$,
\begin{equation}\label{eq:R2}
\|R_2 f_0\|_{L^q(r^{d-1} \mathrm{d}r)}{\lesssim}
\| f_0\|_{L^p(s^{d-1} \mathrm{d}s)}.
\end{equation}
Furthermore, the inequality also holds for $p=1$ and $q<pd$.
\end{proposition}

\begin{proof}
By Hölder's inequality, we have the pointwise estimate
\begin{align}
|R_2 f_0(r)|
&{\lesssim} \sup_{\substack{t\in [1,2]\\t\ge 3r/2}}\Big(r^{-1}\int_{t-r}^{t+r} |f_0(s)|^p s^{d-1} {{\text{\,\rm d}}} s\Big)^{1/p} \notag \\
&{\lesssim} \Big( r^{-1}\int_{1/4}^4 |f_0(s)|^p s^{d-1} {{\text{\,\rm d}}} s\Big)^{1/p}. \label{eq:R2 pointwise}
\end{align}
Define the measure $\mu_d$ on ${{\mathbb{R}}}^+$ by ${{\text{\,\rm d}}}\mu_d=r^{d-1} {{\text{\,\rm d}}} r$. Direct integration in $L^q(\mathrm{d}\mu_d)$ gives \eqref{eq:R2} for $1\le p < \infty$ and $q < pd$. To get the endpoint, we first prove a weak-type inequality.
By \eqref{eq:R2 pointwise} we have that for some $C>1$
\begin{align*}
\mu_d\big (\{r: |R_2 f_0(r)|>\alpha\}) & \le \mu_d\big(\big\{r:
r^{-1} \int_{1/4}^4 |f_0(s)|^p s^{d-1} {{\text{\,\rm d}}} s
>(C^{-1} \alpha)^p \big\}\big)
\\
&=\frac{1}{d} C^{pd}\alpha^{-pd} \Big( \int_{1/4}^4 |f_0(s)|^p s^{d-1} {{\text{\,\rm d}}} s\Big)^{d},
\end{align*}
which shows that $R_2$ maps $L^p(\mu_d)$ to $L^{q,\infty}(\mu_d)$ for $q=pd$. Applying the Marcinkiewicz interpolation theorem with different values of $p$, this upgrades to the $L^p(\mu_d)\to L^q(\mu_d)$ inequality for $1 < p < \infty$.
\end{proof}
We now estimate the operator ${{\mathfrak{M}}}_p$. The interesting $p$-regime is
for $1 \leq p \leq \frac{d}{d-1}$, since for $p>\frac{d}{d-1}$ one can obtain bounds rather trivially, uniformly in $E$.

\begin{proposition}\label{lem:trivial holder}
Let $E\subset [1,2]$. For all $p > \frac{d}{d-1}$ and $1 \leq q \leq \infty$,
\begin{equation*}
\|{{\mathfrak{M}}}_p g\|_{L^q( r^{d-1} \mathrm{d}r)} \lesssim \| g \|_{p}.
\end{equation*}
\end{proposition}

\begin{proof}
We note that ${{\mathfrak{M}}}_p g$ is supported in $[2/3, 4]$. Moreover ${{\mathfrak{M}}}_p g= {{\mathfrak{M}}}_p [g{{\mathbbm 1}}_{[0,4]}]$. Since $s^{d-1-p'}{{\mathbbm 1}} _{[0,4]}$ belongs to $L^{p'}(ds)$ for $p>\frac{d}{d-1}$ the pointwise bound
$|{{\mathfrak{M}}}_p g(r)|{\lesssim} \|g\|_p$ follows from H\"older's inequality. This implies the $L^p\to L^q$ bounds, by the stated support properties.
\end{proof}

For the interesting range $p\le \frac{d}{d-1}$ we distinguish the cases $1 \leq p < \frac{d}{d-1}$, and $p=\frac{d}{d-1}$.

\begin{proposition} \label{lem:rightlineendpt}
Let $1\le p<\frac{d}{d-1}$.
For all $p\le q<\infty$,
\begin{equation*}
\|{{\mathfrak{M}}}_p g\|_{L^q(r^{d-1} {{\text{\,\rm d}}} r)} \le C(p,d)
\sup_{\delta<1}
\delta^{d-1-\frac dp+\frac 1q}
N(E,\delta)^{\frac 1q} \| g \|_{p},
\end{equation*}
where $C(p,d){\lesssim} \tfrac{p}{d-p(d-1)}.$

\end{proposition}

\begin{proof} We argue as in \cite{SeegerWaingerWright1997}. We may assume
$N(E, \delta)\ge 1$ and $d-1-\frac dp+\frac 1q\ge 0$, i.e.
$q\le \frac{p}{d-p(d-1)}$
; otherwise there is nothing to prove.
Since $p<\frac{d}{d-1}$ and $p\le q$ we see that $ \dim_{\mathrm M} E\le
q(\frac{d-1}{p'} +\frac 1q-\frac 1p)= 1- qd(\frac 1p -\frac{d-1}d)<1. $
Observe
that ${{\mathfrak{M}}}_p g(r)=0$ unless $ \frac{2}{3} \leq r \leq 4$.

For each $n \geq 0$, let
\Be\label{eq:UnDn}
U_n:=\{r\in [\tfrac 23,4]: {{\text{\rm dist}}}(r, E)\le 2^{-n+1}\}, \qquad D_n=U_n\setminus U_{n+1}
\Ee
Since ${
\mathrm{dim}}_{\mathrm M} E<1$ the set $\overline E=\bigcap_{n=0}^{\infty} U_n=
[\tfrac 23,4] \setminus \bigcup_{n=0}^\infty D_n$ is of measure zero.
If $r \in D_n$, then
\begin{align*}
{{\mathfrak{M}}}_p g(r) & \lesssim \int_{2^{-n}}^{6} s^{\frac{d-1}{p'} - 1} |g(s)| {{\text{\,\rm d}}} s \lesssim \sum_{\ell = 0}^{n+2} 2^{(-n+ \ell)(\frac{d-1}{p'} - 1)} \int_{2^{-n+\ell}}^{2^{-n+ \ell + 1}} |g(s)| {{\text{\,\rm d}}} s \\
& \lesssim \sum_{\ell = 0}^{n+2} 2^{(-n+ \ell)(\frac{d-1}{p'} - 1)} 2^{(-n+\ell)/p'} \Big( \int_{2^{-n+\ell}}^{2^{-n+ \ell + 1}} |g(s)|^p {{\text{\,\rm d}}} s \Big)^{\frac{1}{p}}.
\end{align*}

We have from the above observations and Minkowski's inequality that
\begin{align*}
&\| {{\mathfrak{M}}}_p g \|_{L^q(r^{d-1} \mathrm{d} r)} \lesssim \Big( \sum_{n \geq 0} \int_{D_n} |{{\mathfrak{M}}}_p g(r)|^q {{\text{\,\rm d}}} r \Big)^{1/q}\\
& \lesssim \Big(\sum_{n\ge 0}
|D_n| \Big[ \sum_{\ell=0}^{n+2} 2^{(\ell-n)(d-1-\frac dp)}\Big( \int_{2^{\ell-n}}^{2^{\ell+1-n}} |g(s)|^p {{\text{\,\rm d}}} s \Big)^{\frac{1}{p}} \Big]^q \Big)^{\frac 1q}
\end{align*}
and by Minkowski's inequality and the embedding $\ell^p \subset \ell^q$ this is
\begin{align*}
&{\lesssim}
\sum_{\ell\ge 0}
2^{-\ell(\frac dp-d+1)}
\Big(\sum_{ \substack{n \in {\mathbb N}_0: \\ n\ge \ell-2}} \Big[
2^{n(\frac dp-d+1)}
|D_n|^{\frac 1q}
\Big(\int_{2^{-n+\ell}}^{2^{-n+\ell+1}} |g(s)|^p {{\text{\,\rm d}}} s\Big)^{\frac 1p}
\Big]^q \Big)^{\frac 1q}
\\&{\lesssim}
\sum_{\ell\ge 0} 2^{-\ell(\frac dp-d+1)}
\Big(\sum_{n\ge 0} 2^{n(\frac dp-d+1)p}
|D_n|^{\frac pq}
\int_{2^{-n+\ell}}^{2^{-n+\ell+1}} |g(s)|^p {{\text{\,\rm d}}} s \Big)^{\frac 1p}.
\end{align*}

We may estimate, for each fixed $\ell$, the $n$-sum by
$ \sup_{n \geq 0} 2^{n(\frac dp-d+1)}
|D_n|^{1/q} \|g\|_p$.
Since $p<\frac{d}{d-1}$ we may then sum in $\ell$ and use
$|D_n|\lesssim N(E, 2^{-n}) 2^{-n}$
to get
\begin{align*}
\| {{\mathfrak{M}}}_p g \|_{L^q(r^{d-1} {{\text{\,\rm d}}} r)} &\lesssim
\tfrac{1}{1-2^{ d-1-\frac dp}}
\sup_{n \geq 0} 2^{n(\frac dp-d+1)}
|D_n|^{1/q} \|g\|_p,
\\&\lesssim \tfrac{p}{d-p(d-1)}\,
\sup_{n \geq 0} 2^{n(\frac dp-d+1-\frac 1q)}
N(E,2^{-n})^{\frac 1q} \|g\|_p,
\end{align*}
which concludes the proof.
\end{proof}

We argue in a slightly different way for the endpoint $p=\frac{d}{d-1}$.

\begin{proposition} \label{lem:rightlineendpt=beta=1}
Let $E\subset [1,2]$, and $p_d:=\frac{d}{d-1}$.
Then for $q\ge p_d$,
\begin{equation}\label{LpLq-beta=1}
\|{{\mathfrak{M}}}_p g \|_{L^q (r^{d-1} \mathrm{d}r)} {\lesssim} \sup_{\delta<1} \delta^{\frac{1}{q}}N(E,\delta)^{\frac 1q} (\log \tfrac 1\delta) ^{\frac 1d} \| g \|_{p_d}.
\end{equation}
\end{proposition}

\begin{proof}
We may assume $\sup_{\delta<\frac 12} \delta^{\frac{1}{q}}N(E,\delta)^{\frac 1q} (\log \tfrac 1\delta) ^{\frac 1d}<\infty$. Then, if for $n \geq 1$ we consider the sets $U_n$, $D_n$ as in \eqref{eq:UnDn} we have that $|U_n|{\lesssim} n^{-q/d}$. Recalling that $\bar E=\bigcap_{n=0}^\infty U_n$, this implies that $\overline E$ is of measure zero.
Note that ${{\mathfrak{M}}}_pg(r)=0$ unless $\frac{2}{3} \leq r \leq 4$.
We set, for $\ell>0$, \[\Omega_\ell:=\bigcup_{2^{\ell-1}\le n<2^{\ell}} D_n.\]
Note that
\begin{equation}\label{eq:Omell-est}
|\Omega_\ell| {\lesssim} |\{r:{{\text{\rm dist}}}(r,E){\lesssim} 2^{-2^{\ell-1}} \}| {\lesssim} N(E, 2^{-2^{\ell-1}}) 2^{-2^{\ell-1}} {\lesssim} B^q 2^{-\ell q/d},
\end{equation}
where $B:= \sup_{\delta<1} \delta^{\frac{1}{q}}N(E,\delta)^{\frac 1q} (\log \frac 1\delta) ^{\frac 1d} \gtrsim 1$; here we assume without loss of generality $E \neq \emptyset$.
We estimate
${{\mathfrak{M}}}_p g(r){\lesssim} {{\mathfrak{N}}}_{p} g(r) +{{\mathcal{E}}} g(r) $ where
\begin{align*}
{{\mathfrak{N}}}_{p} g(r)&:= \sum_{n>0}{{\mathbbm 1}}_{D_n}(r) \int_{2^{-n}}^{1/2} s^{-1/d} g(s) {{\text{\,\rm d}}} s,
\\
{{\mathcal{E}}} g(r) & {\lesssim} {{\mathbbm 1}}_{[1/3,2]}(r) \int_{1/2}^4 s^{-1/d} g(s) {{\text{\,\rm d}}} s.
\end{align*}
For the term ${{\mathcal{E}}}$ we just use Hölder's inequality and obtain
\begin{equation*}
\|{{\mathcal{E}}} g \|_{L^q(r^{d-1} \mathrm{d}r)} {\lesssim}
\Big(\int_{1/2}^4 \Big[\int_{1/2}^4 s^{-1/d} g(s) {{\text{\,\rm d}}} s\Big]^q {{\text{\,\rm d}}} r\Big)^{1/q} {\lesssim} \|g\|_{p_d}.
\end{equation*}

We turn to the term ${{\mathfrak{N}}}_{p} g(r) $. For $m\ge 0$, let $J_m=[2^{-2^{m+1}}, 2^{-2^{m}}]$. Then, if $r \in D_n$, with $2^{\ell-1} \le n< 2^{\ell}$, we have by Hölder's inequality
\begin{align*} |{{\mathfrak{N}}}_{p} g(r)|
& \lesssim \int_{ 2^{-2^{\ell}}}^{1/2} s^{-1/d} |g(s)| {{\text{\,\rm d}}} s
\le \sum_{k=0}^\ell\int_{J_{\ell-k}}
s^{-1/d} |g(s)|{{\text{\,\rm d}}} s\\ &
\le \sum_{k=0}^{\ell} \Big(\int_{J_{\ell-k}} s^{-1} {{\text{\,\rm d}}} s\Big)^{1/d} \Big(\int_{J_{\ell-k}}|g(s)|^{\frac d{d-1}} {{\text{\,\rm d}}} s\Big)^{\frac{d-1}{d}}
\\&{\lesssim}
\sum_{k=0}^{\ell} 2^{ (\ell-k)/d} \Big(\int_{J_{\ell-k}}|g(s)|^{p} {{\text{\,\rm d}}} s\Big)^{1/p}.
\end{align*}
Consequently,
\begin{align*}
\|{{\mathfrak{N}}}_{p} g\|_{L^q(r^{d-1}\mathrm{d}r)}&{\lesssim}
\Big(\sum_{\ell=1}^\infty \sum_{2^{\ell-1} \leq n < 2^{\ell}} \int_{D_n} |{{\mathfrak{N}}}_{p}g(r)|^q {{\text{\,\rm d}}} r\Big)^{1/q}
\\&
{\lesssim}\Big(\sum_{\ell=1}^\infty |{\Omega}_\ell| \Big(\sum_{k=0}^\ell 2^{(\ell-k)/d}\Big (\int_{J_{\ell-k}}|g(s)|^p {{\text{\,\rm d}}} s\Big)^{1/p }\Big)^q\Big) ^{ 1/q}.
\end{align*}
By the triangle inequality in $\ell^{q}$ the above right-hand side is
\begin{align*}
&{\lesssim}\sum_{k=0}^\infty 2^{-k/d} \Big(\sum_{\ell=k}^\infty |{\Omega}_\ell| 2^{\ell q/d} \Big(\int_{J_{\ell-k}}|g|^p{{\text{\,\rm d}}} s\Big)^{q/p}\Big)^{1/q}
\\&{\lesssim} B
\sum_{k=0}^\infty 2^{-k/d} \Big(\sum_{\ell=k}^\infty \Big(\int_{J_{\ell-k}}|g|^p {{\text{\,\rm d}}} s\Big)^{q/p}\Big)^{1/q}
\\
&{\lesssim} B
\sum_{k=0}^\infty 2^{-k/d} \Big(\sum_{\ell=k}^\infty \int_{J_{\ell-k}}|g|^p {{\text{\,\rm d}}} s\Big)^{1/p}
{\lesssim} B\|g\|_p;
\end{align*}
here we have used \eqref{eq:Omell-est} in the second inequality, the embedding $\ell^p\hookrightarrow\ell^q$ for $p\le q$ in the third and the disjointness of the intervals $J_{\ell-k}$ for $\ell\ge k$ and fixed $k \geq 0$. This establishes the bound $\| {{\mathfrak{N}}}_{p} g \|_{L^q(r^{d-1} \mathrm{d}r)} \lesssim B \|g \|_{p}$, which concludes the proof of \eqref{LpLq-beta=1}.
\end{proof}

We are now in a position to conclude Theorem \ref{thm:highdim}.

\begin{proof}[Proof of Theorem \ref{thm:highdim}]
We recall that the triangle $\Delta_\beta$ is given by
\begin{equation*}
\Delta_\beta= \big\{(\tfrac{1}{p},\tfrac{1}{q}) \in [0,1]^2 : \tfrac{1}{pd} \leq \tfrac{1}{q} \leq \tfrac{1}{p}, \,\,\, \tfrac{1-\beta}{q} + d-1 \geq \tfrac{d}{p} \big\},
\end{equation*}
and note that the line segment $[P_{2,\beta}, P_{3,\beta}^{\mathrm{rad}}]$ satisfies the second condition with equality. From Corollary \ref{cor: Delta beta nec} we have ${{\mathcal{T}}}_E^{\mathrm{rad}} \subset \Delta_\beta$. For the sufficient conditions, we have by Propositions \ref{Ronelem} and \ref{lem:Rtwo} that if $(\frac{1}{p}, \frac{1}{q}) \in \Delta_\beta$ then $R_1$ and $R_2$ are $L^p(s^{d-1} \mathrm{d}s) \to L^q(r^{d-1} \mathrm{d}r)$ bounded. Thus, by Lemma \ref{lem:decomposition d3}, we shall only focus on ${{\mathfrak{M}}}_p$.

\medskip
(i) Let $\beta<1$. If $(\frac{1}{p},\frac{1}{q}) \in [P_{2,\beta},P_{3,\beta}^{{\text{\rm rad}}}]$ and $\beta <1$ then $1 \leq p<\frac{d}{d-1}$. We can then apply Proposition \ref{lem:rightlineendpt} together with the assumption $\sup_{\delta<1} \delta^{\beta} N(E,\delta) <\infty$ to deduce that ${{\mathfrak{M}}}_p$ is $L^p \to L^q(r^{d-1}\mathrm{d}r)$ bounded on that line segment. The remaining bounds for $(\frac{1}{p}, \frac{1}{q}) \in \Delta_\beta$ follow by interpolation with the case $q=p=\infty$. Therefore ${{\mathcal{T}}}_E^{{\text{\rm rad}}}=\Delta_\beta$.

\medskip

(ii) If $\beta <1$ and $\sup_{\delta < 1} \delta^{\beta}N(E,\delta)=\infty$,
we have by Lemma \ref{lem: nec easy} that \[\| M_E \|_{L^p_{\mathrm{rad}} \to L^q } \gtrsim \sup_{\delta < 1} N(E,\delta)^{1/q} \delta^{\beta/q} = \infty \text{ if } (\tfrac{1}{p}, \tfrac{1}{q}) \in [P_{2,\beta}, P_{3,\beta}^{{{\text{\rm rad}}}}].\]
For the positive bounds, we note by Proposition \ref{lem:trivial holder} that ${{\mathfrak{M}}}_p$ is $L^p \to L^q(r^{d-1}\mathrm{d}r)$ bounded for $p > \frac{d}{d-1}$. For $p \leq \frac{d}{d-1}$, we use that for every $\varepsilon>0$ we have
$N(E,\delta) \lesssim_\varepsilon \delta^{-\beta-\varepsilon}
$ uniformly in $\delta>0$. Since $\beta <1$, Proposition \ref{lem:rightlineendpt=beta=1} guarantees $L^{p_d} \to L^q(r^{d-1}\mathrm{d}r)$ boundedness for $q \geq p_d$ for a choice of $\varepsilon>0$ sufficiently small. Finally, note that
\begin{equation}\label{eq: delta mink easy}
\delta^{d-1-\frac{d}{p} + \frac{1}{q}}N(E,\delta)^{1/q} \lesssim_\varepsilon \delta^{d-1-\frac{d}{p} + \frac{1-\beta-\varepsilon}{q}}.
\end{equation}
For $p< \frac{d}{d-1}$ and $(\frac{1}{p}, \frac{1}{q}) \in \Delta_\beta \backslash [P_{2,\beta}, P_{3,\beta}^{{\text{\rm rad}}}]$, we have that $\frac{1-\beta}{q} + d-1 > \frac{d}{p}$. Thus, choosing a sufficiently small $\varepsilon>0$, the right-hand side of \eqref{eq: delta mink easy}
is uniformly bounded in $0 <\delta < 1$. We can then apply Proposition \ref{lem:rightlineendpt} to deduce that ${{\mathfrak{M}}}_p$ is $L^p \to L^q(r^{d-1}\mathrm{d}r)$ bounded in this case. Consequently, we have proven ${{\mathcal{T}}}_E^{\mathrm{rad}}= \Delta_\beta \backslash[P_{2,\beta}, P_{3,\beta}^{\mathrm{rad}}]$.

\medskip

(iii) Assume that ${\dim}_M E =1$. We already noted that ${{\mathcal{T}}}_E^{\mathrm{rad}}\subset \Delta_1$. Moreover, for $p_d=\frac{d}{d-1}$, we have from Lemma
\ref{lem:lowerbd-pd} that $\sup_{ \delta < 1} \delta (\log \frac{1}{\delta})^{\frac{q}{d}} N(E,\delta)$ must be finite if $\| M_E\|_{L^{p_d}_{\mathrm{rad}} \to L^q} \lesssim 1$. This establishes the $\subset$ implication. For the sufficient condition, we note that $P_{2,1}=(\frac{d-1}{d},\frac{d-1}{d})$ and $P_{3,1}^{\mathrm{rad}}=(\frac{d-1}{d},\frac{d-1}{d^2})$. If $(\frac{1}{p}, \frac{1}{q}) \in \Delta_1$ with $p > \frac{d}{d-1}$, $L^p \to L^q(r^{d-1}\mathrm{d}r)$ bounds immediately follow from Proposition \ref{lem:trivial holder}. On the other hand, if $p=\frac{d}{d-1}$, the bounds follow from Proposition \ref{lem:rightlineendpt=beta=1} under the assumption $\sup_{\delta<\frac 12} \delta N(E,\delta)(\log \tfrac 1\delta) ^{\frac qd}< \infty$. This establishes the $\supseteq$ implication, which concludes the proof.
\end{proof}

\section{2-dimensional results}\label{sec:2d}
In this section we consider the circular maximal function $M_E$ for radial functions on the plane. As stated in \cite[Lemma 5.1]{SeegerWaingerWright1997}
the formulas \eqref{eq:polar} and \eqref{eq:Kt-polar} yield by straightforward estimation a pointwise inequality which involves kernels that are more singular than their higher-dimensional counterparts. In what follows, for a radial function $f$, we continue to use the notation $f(x)=f_0(s)$ for $s=|x|$.

\begin{lemma}\label{lem:2d reduction}
Let $d=2$. Fix $1\le p<\infty$ and set
$g(s):=f_0(s)s^{1/p}$. Then
\begin{equation*}
M_E f(x)\lesssim {{\mathfrak{M}}}_p^+ g(r)+ {{\mathfrak{M}}}_p^- g(r)+\sum_\pm \sum_{i=1}^2 R_i^\pm f_0(r)
\end{equation*}
where $r=|x|$ and
\begin{align*}
{{\mathfrak{M}}}_p^- g(r)&:=\sup_{\substack{ t\in E\\r/2<t<3r/2}}
r^{-1}\Big|\int_{|r-t|}^{r+t}s^{ 1/2- 1/p}(s-|r-t|)^{-1/2} g(s){{\text{\,\rm d}}} s\Big|,
\\
{{\mathfrak{M}}}_p^+ g(r)&:=\sup_{\substack{ t\in E\\r/2<t<3r/2}}
r^{-1} \Big| \int_{|r-t|}^{r+t}s^{ 1/2- 1/p}(r+t-s)^{-1/2} g(s){{\text{\,\rm d}}} s\Big|,
\end{align*}
\begin{align*}
R_1^{-} f_0(r)\,&:=\,\sup_{\substack{t\in E\\ t\le r/2}} t^{-1/2}\Big|\int_{r-t}^{r}
|s-r+t|^{-1/2} f_0(s) {{\text{\,\rm d}}} s \Big|,
\\
R_1^+ f_0(r)\,&:=\,\sup_{\substack{ t\in E\\ t\le r/2}} t^{-1/2}\Big|\int_{r}^{r+t}
|r+t-s|^{-1/2} f_0(s) {{\text{\,\rm d}}} s \Big|,
\end{align*} and
\begin{align*}
R_2^- f_0(r)\,&:=\,\sup_{\substack{ t\in E\\ t\ge 3r/2}} r^{-1/2}\Big| \int_{t-r}^{t}
|s-t+r|^{-1/2} f_0(s) {{\text{\,\rm d}}} s \Big|,
\\
R_2^+ f_0(r)\,&:=\,\sup_{\substack{t\in E\\t\ge 3r/2}} r^{-1/2}\Big|\int_{t}^{t+r}
|r+t-s|^{-1/2} f_0(s) {{\text{\,\rm d}}} s\Big|.
\end{align*}
\end{lemma}

Given this lemma, it suffices to establish Lebesgue space bounds for the operators $R_1^\pm, R_2^\pm$ and ${{\mathfrak{M}}}_p^\pm$.

\subsection{The operators \texorpdfstring{$R_1^\pm$}{R1}}
Boundedness for these operators hold under a condition involving the Minkowski dimension. The argument is analogous to that in \cite[Proposition 5.2]{SeegerWaingerWright1997}.

\begin{proposition}\label{prop:R1R2}
For all $1 < p \leq q \leq \infty$, we have
\begin{equation*}
\| R_1^\pm f_0 \|_{L^q(r\mathrm{d}r)}
\lesssim \sum_{m \geq 0} 2^{-m(\frac{1}{2} + \frac{1}{q} - \frac{1}{p})} N(E,2^{-m})^{\frac 1q} \| f_0 \|_{L^p(s \mathrm{d}s)}.
\end{equation*}
\end{proposition}

\begin{proof} We only give the estimate for $R_1:=R_1^-$; the same arguments apply to $R_1^+$.
First, note that without loss of generality, we can assume that $f_0$ is supported on $[1,\infty)$. For each $r \geq 2$ we can write, after a dyadic decomposition,
\begin{equation}\label{eq:R_1 dec}
R_1 f_0 (r) \lesssim M_{HL}^{\mathrm{loc}} f_0(r) + \sum_{m > 0} 2^{m/2} R_1^{m} f_0(s)
\end{equation}
where
\begin{equation*}
R_1^{m} f_0(r):=\sup_{\substack{t \in E \\ t \leq r/2}} \int_{r-t + 2^{-m}}^{r-t+2^{-m+1}} |f_0(s)| {{\text{\,\rm d}}} s
\end{equation*}
and $M_{HL}^{\mathrm{loc}}h(r):=\sup_{1\leq t \leq 2} \frac{1}{t}\int_{r-t}^r |h(s)|{{\text{\,\rm d}}} s$. Clearly \[ \|{{\mathbbm 1}}_{(2,\infty)}M_{HL}^\mathrm{loc} f_0\|_{L^q(r\mathrm{d}r)} \lesssim \|f_0\|_{L^p(r\mathrm{d}r)} \] for all $1 < p \leq q \leq \infty$. For each fixed $m>0$, let
$\{I_m^\nu: \nu \in {{\mathfrak{N}}}_m(E)\}$
be a minimal cover of $E$ by intervals of length $2^{-m}$,
indexed by the set ${{\mathfrak{N}}}_m(E)$ with
$\#{{\mathfrak{N}}}_m(E)=N(E,2^{-m})$. For each $I_m^\nu$, define $J_m^\nu:= \{t \in {\mathbb R} : {{\text{\rm dist}}}(t, I_m^\nu) \leq 2^{-m+1}\}$, that is, its concentric interval with 5 times the length. Then
\begin{equation*}
R_1^{m} f_0(r) \lesssim \sup_{\nu \in {{\mathfrak{N}}}_m(E)} \int_{J_m^\nu} |f_0(r-s)| {{\text{\,\rm d}}} s = \sup_{\nu \in {{\mathfrak{N}}}_m(E)} \, |f_0 {{\mathbbm 1}}_{A_k}| \ast {{\mathbbm 1}}_{J_m^\nu}(r)
\end{equation*}
whenever $r \in [2^k,2^{k+1}]$ for $k>0$, where $A_k:=[\max\{1, 2^{k}-3\}, 2^{k+1}]$.
Consequently, by Young's convolution inequality,
\begin{align*}
\int_{2}^\infty &|R_1^{m} f_0(r)|^q r {{\text{\,\rm d}}} r \lesssim \sum_{\nu \in {{\mathfrak{N}}}_m(E)} \sum_{k > 0} 2^k \int_{2^k}^{2^{k+1}} \big| |f_0 {{\mathbbm 1}}_{A_k}| \ast {{\mathbbm 1}}_{J_m^\nu} (r)\big|^q {{\text{\,\rm d}}} r \\
& \lesssim \sum_{\nu \in {{\mathfrak{N}}}_m(E)} \sum_{k \geq 0} 2^k |I_m^\nu|^{q + 1 - \frac{q}{p}} \Big( \int_{A_k} |f_0(s)|^p {{\text{\,\rm d}}} s \Big)^{q/p} \\
& \lesssim 2^{-m(q +1 - \frac{q}{p})} N(E,2^{-m}) \sup_{k \geq 0} 2^{k(1 - \frac{q}{p})} \sum_{k \geq 0} \Big( \int_{2^{k-2}}^{2^{k+1}} |f_0(s)|^p s {{\text{\,\rm d}}} s \Big)^{q/p} \\
& \lesssim 2^{-m(q +1 - \frac{q}{p})} N(E,2^{-m}) \| f_0 \|_{L^p(s\mathrm{d}s)}^q,
\end{align*}
where the last inequality follows since $p \leq q$. Combining this with \eqref{eq:R_1 dec} concludes the proof.
\end{proof}

\subsection{The operators \texorpdfstring{$R_2^\pm$}{R2}}
We first record a result from \cite{SeegerWaingerWright1997} for the case $p=q$.
\begin{lemma}\label{lem:R2-p=q}
For all $1\le p< \infty$, we have
\[\|R_2^\pm f\|_{L^p(r\mathrm{d} r)}\le \sum_{m\ge 0} N(E, 2^{-m})^{1/p} 2^{-m/2} \|f\|_{L^p(s\mathrm{d} s)}.\]
\end{lemma}
For the proof see \cite[Prop. 5.3]{SeegerWaingerWright1997}. The $L^p(s\mathrm{d} s)\to L^q(r\mathrm{d} r)$ estimates for $R_2^\pm$
are more involved, and can be obtained under a condition on the quantity $\nu^\sharp(\alpha)$ in \eqref{eq:nudagger}. We note that the behaviour of $R_2^\pm$ is responsible for the different outcomes in two versus higher dimensions.

\begin{proposition}\label{prop:R3-better}
Let $1 < p \leq q <\infty$ and $E \subset [1,2]$. For $k,m \geq 0$, define
\begin{equation}\label{eq:omega def}
\omega_m^{p,q}(E,k):= \sup_{|J|=2^{-k}}2^{-k(\frac{2}{q}-\frac{1}{p})} N(E \cap J, 2^{-m-k})^{\frac{1}{q}}
\end{equation}
where the supremum is taken over all intervals $J \subset [1,2]$ of length $2^{-k}$.
\begin{enumerate}[(i)]
\item For $2 \leq q \leq 2p$, we have
\begin{equation}\label{eq:R3 better 2}
\| R_2^\pm f_0 \|_{L^q(r\mathrm{d}r)} \lesssim \sum_{m \geq 0} 2^{-m(\frac{1}{2} + \frac{1}{q} - \frac{1}{p})} \| \omega_m^{p,q}(E,k)\|_{\ell^\infty_k} \| f_0 \|_{L^p(s \mathrm{d}s)}.
\end{equation}
\item If $1 < q < 2$, we have
\begin{equation}\label{eq:R3 better 1}
\| R_2^\pm f_0 \|_{L^q(r\mathrm{d}r)} \lesssim
\sum_{m \geq 0} 2^{-m(\frac{1}{2} + \frac{1}{q} - \frac{1}{p})} \|\omega_m^{p,q}(E,k) \|_{\ell^{\frac{2q}{2-q}}_k} \| f_0 \|_{L^p(s \mathrm{d}s)}.
\end{equation}
\end{enumerate}
\end{proposition}

\begin{proof} We only give the estimate for $R_2:=R_2^{-}$; the same arguments apply to $R_2^+$. Note that $R_2 f_0(r)=0$ for $r \geq 4/3$. Furthermore, without loss of generality, we may assume that $f_0$ is supported on $[1/3,2]$. Let $I_0:=[1/2, 4/3]$ and, for $k> 0$, let $I_k:=[2^{-k-1}, 2^{-k}]$. Let $\chi_0\in C^\infty_c$ be even, such that $\chi_0(y)=1$ for $|y|\le 2^{-12}$ and $\chi_0(y)=0$ for $|y|\ge 2^{-11}.$ For $k \geq 0$, we define
\begin{align*}
{\mathscr{A}}_k f_0(r,t) &:= r^{-1/2}{{\mathbbm 1}}_{I_k}(r) \int_{t-r}^\infty \chi_0(2^{k} (t-r-s)) |t-r-s|^{-1/2} f_0(s) {{\text{\,\rm d}}} s
\\
{\mathscr{B}}_k f_0(r,t)&:= r^{-1/2}{{\mathbbm 1}}_{I_k}(r)\int_{t-r}^t (1-\chi_0(2^{k} (t-r-s))) |t-r-s|^{-1/2} f_0(s) {{\text{\,\rm d}}} s
\end{align*}
and note that
\begin{equation}\label{eq:R3 A B}
R_2 f_0(r) \leq \sup_{t\in E} \sum_{k\ge 0} |{\mathscr{A}}_{k} f_0(r,t)| \,+\,
\sup_{t\in E} \sum_{k\ge 0} |{\mathscr{B}}_{k} f_0(r,t) |.
\end{equation}
We start bounding the ${\mathscr{B}}_k$ terms. For $t\in E$ and $r\in I_k $ we have by Hölder's inequality
\begin{equation*}
|{\mathscr{B}}_k f_0(r,t) |{\lesssim} 2^k \Big(\int_{t-r}^t|f_0(s)|{{\text{\,\rm d}}} s \Big) {{\mathbbm 1}}_{I_k}(r) {\lesssim} 2^{k/p} \Big(\int_{1/3}^2|f_0(s)|^p {{\text{\,\rm d}}} s\Big)^{1/p} {{\mathbbm 1}}_{I_k}(r),
\end{equation*}
and thus
\begin{align*}
\sup_{1\le t\le 2} \sum_{k\ge 0} |{\mathscr{B}}_k f_0(r,t)| {\lesssim} r^{-2/q} \sup_{k \geq 0} 2^{-k(2/q-1/p)} \|f_0\|_p
{\lesssim} r^{-2/q} \|f_0\|_p
\end{align*}
for $q \leq 2p$.
Since $r^{-2/q}\in L^{q,\infty}(r \mathrm{d}r)$ we have
$L^p(s \mathrm{d}s)\to L^{q,\infty}(r \mathrm{d}r)$ bounds for $q \leq 2p$. Therefore, by the Marcinkiewicz interpolation theorem, we obtain the strong-type bounds
\begin{equation}\label{eq:sumBkest} \Big\|\sup_{1\le t\le 2} \sum_{k\ge 0} |{\mathscr{B}}_k f_0(\cdot,t)|\Big\|_{L^q(r \mathrm{d}r)} {\lesssim} \|f_0\|_{L^p(s \mathrm{d}s)},
\end{equation}
for $1 < p < \infty$ and $q \leq 2p$.

We turn to the ${\mathscr{A}}_k$ terms. Let
\[ h_k(x)=\begin{cases} 0 &\text{ if } x>0,\\
\chi_0(2^k x)|x|^{-1/2} &\text{ if } x\le 0.
\end{cases}
\]
In particular, $h_k$ is supported on $(-2^{-11-k},0)$. Note that
\begin{equation*}
{\mathscr{A}}_k f_0(r,t)=r^{-1/2} {{\mathbbm 1}}_{I_k}(r) \,\, h_k*f_0(t-r).
\end{equation*}
We perform a further decomposition of the operator ${\mathscr{A}}_k$ which will quantify the size of $|t-r-s|^{-1/2}$. To this end, we use, for $k \geq 0$, the resolutions of the identity
\begin{equation}\label{eq:deltaresmain}
\delta= u_k + \sum_{m=1}^\infty \upsilon_{k+m} * \psi_{k+m}
\end{equation}
where
$u_k=2^ku(2^k\cdot)$, $\upsilon_{\ell}=2^{\ell-1} \upsilon_{1}(2^{\ell-1}\cdot)$, $\psi_{\ell}=2^{\ell-1}\psi_1(2^{\ell-1}\cdot)$
and
$u, \upsilon_1, \psi_1$
are $C_c^\infty$ functions supported on $(-2^{-10}, 2^{-10})$ and with $\upsilon_{1}$, $\psi_{1}$ satisfying moment conditions up to a certain fixed order $N>0$. See \cite[Lemma 2.1]{se-tao} for a proof.
The convergence in \eqref{eq:deltaresmain} is in the sense of tempered distributions.
This resolution of the identity is convenient since it will allow for an application of Littlewood--Paley theory to sum in the $k$-variable. With this at our disposal, define
\begin{align*}
{\mathscr{A}}_{k,0} f_0(r,t)&:=r^{-1/2}{{\mathbbm 1}}_{I_k}(r) h_k*u_k*f_0(t-r),
\\
{\mathscr{A}}_{k,m} f_0(r,t)&:= r^{-1/2} {{\mathbbm 1}}_{I_k}(r) h_k*\upsilon_{k+m}*\psi_{k+m} *f_0(t-r), \qquad m\ge 1.
\end{align*}
In view of \eqref{eq:deltaresmain}
we have ${\mathscr{A}}_k=\sum_{m=0}^\infty {\mathscr{A}}_{k,m} $ for all $k\ge 0$.

Consider first the ${\mathscr{A}}_{k,0}$. Note that $h_k*u_k$ is supported in $[-2^{-k-10},2^{-k-20}]$ and $\|h_k*u_k\|_\infty \leq \| h_k \|_1 \| u_k\|_\infty {\lesssim} 2^{k/2}$. Thus, one can argue exactly as for the ${\mathscr{B}}_k$ terms and deduce the estimate
\begin{equation}\label{eq:sumA0kest}
\Big\| \sup_{1\le t\le 2} \sum_{k\ge 0} |{\mathscr{A}}_{k,0} f_0(\cdot,t)|\Big\|_{L^q(r \mathrm{d}r)} {\lesssim} \|f_0\|_{L^p(s \mathrm{d}s)}
\end{equation}
for $1<p<\infty$ and $q\le 2p$.

We turn to the operators ${\mathscr{A}}_{k,m}$ for $m \geq 1$, $k \geq0$. Note that by \eqref{eq:R3 A B}, \eqref{eq:sumBkest} and \eqref{eq:sumA0kest}, to conclude the proof of the proposition it suffices to show that the operator $f_0 \mapsto \sup_{t \in E} \sum_{k \geq 0} |\sum_{m \geq 1} {\mathscr{A}}_{k,m} f_0(\cdot, t)|$ satisfies the bounds \eqref{eq:R3 better 2} and \eqref{eq:R3 better 1}. The key estimate for this end is
\Be \label{eq:fixedkAm}
\| \sup_{t \in E} |{\mathscr{A}}_{k,m} f_0(\cdot, t)| \|_{L^q(r \mathrm{d}r)}
{\lesssim} 2^{-m(\frac 12+\frac 1q-\frac 1p)}
\omega_{m}^{p,q}(E, k)\|\psi_{k+m} \ast f_0\|_p
\Ee
for $1<p<\infty$ and $q\le 2p$.
Assuming \eqref{eq:fixedkAm}, one can conclude the proof using a Littlewood--Paley inequality for the family of functions $\{\psi_{k+m}\}_{k\ge 0}$, for any fixed $m\ge 1$.
Because of the disjointness of the intervals $I_k$ and \eqref{eq:fixedkAm}
we have
\begin{align}\notag
&\Big\|\sup_{t\in E} \sum_{k\ge 0}\big| \sum_{m \geq 1} {\mathscr{A}}_{m, k} f_0(\cdot,t) \big|\Big\|_{L^q(r\mathrm{d}r)} {\lesssim} \sum_{m \geq 1} \Big(\sum_{k\ge 0} \big\| \sup_{t\in E} | {\mathscr{A}}_{m,k} f_0(\cdot,t)| \big\|_{L^q(r\mathrm{d}r)}^q\Big)^{\frac 1q}&
\\& \qquad \qquad \label{eq: disjointness}
{\lesssim}
\sum_{m \geq 1} 2^{-m(\frac 12+\frac 1q-\frac 1p)} \Big(\sum_{k\ge 0}
\omega_{m}^{p,q}(E,k)^q \|\psi_{k+m} \ast f_0\|_p^q
\Big)^{\frac 1q}.
\end{align}
In the case $q\ge 2$ we use
\[ \Big(\sum_{k\ge 0}
\omega_{m}^{p,q}(E,k)^q \|\psi_{k+m} \ast f_0\|_p^q
\Big)^{\frac 1q} {\lesssim}
\sup_{k\ge 0} \omega_{m}^{p,q}(E,k) \Big(\sum_{k\ge 0}
\|\psi_{k+m} \ast f_0\|_p^q
\Big)^{\frac 1q},
\] and then by Minkowski's inequality and Littlewood--Paley theory
\begin{align*}
\notag \Big(\sum_{k} &\|\psi_{k+m} *f_0\|_p^q\Big)^{1/q} {\lesssim} \Big\| \Big(\sum_k|\psi_{k+m} *f_0 |^q\Big)^{1/q} \Big\|_p \\
&{\lesssim} \Big\| \Big(\sum_k|\psi_{k+m} *f_0 |^2\Big)^{1/2} \Big\|_p {\lesssim}
\|f_0\|_{p} \sim \| f_0 \|_{L^p(s \mathrm{d}s)}.
\end{align*}
This together with \eqref{eq: disjointness} implies part (i) of the proposition.

For $q\le 2$ we use H\"older's inequality to get, for $\omega(k):=\omega_{m}^{p,q}(E,k)$,
\[ \Big(\sum_{k\ge 0}
\omega(k)^q \|\psi_{k+m} \ast f_0\|_p^q
\Big)^{\frac 1q} {\lesssim} \Big(\sum_{k\ge 0} \omega(k)^{\frac{2q}{2-q}}\Big)^{\frac 1q-\frac 12} \Big(\sum_{k\ge 0}
\|\psi_{k+m} \ast f_0\|_p^2
\Big)^{\frac 12}.\]
By Minkowski's inequality and Littlewood--Paley theory
\[ \Big(\sum_{k} \|\psi_{k+m} *f_0\|_p^2\Big)^{1/2} {\lesssim} \Big\| \Big(\sum_k|\psi_{k+m} *f_0 |^2\Big)^{1/2} \Big\|_p
{\lesssim}
\|f_0\|_{p} \sim \| f_0 \|_{L^p(s \mathrm{d}s)}
\]
and we see that this together with \eqref{eq: disjointness} also implies part (ii) of the proposition.

It remains to prove \eqref{eq:fixedkAm} for fixed $m\ge 1$, $k \geq 0$. As mentioned above, the $m$-decomposition allows to essentially quantify the magnitude of $|t-r-s|$; in practice, one shall think of $h_k \ast \upsilon_{k+m}$ as being roughly $2^{(k+m)/2}{{\mathbbm 1}}_{[-2^{-m-k}, 2^{-m-k}]}$. More precisely, using the vanishing moment conditions of $\upsilon_{k+m}$ one obtains for $x\in {{\mathbb{R}}}$,
\begin{equation}\label{eq:smofhk}
|h_k*\upsilon_{k+m} (x) |
{\lesssim} 2^{(k+m)/2} (1+2^{k+m}|x|)^{-N}
\end{equation}
for all $N \geq 0$. For $|x|\leq 2^{-k-m}$, this follows from pulling out $\| \upsilon_{k+m} \|_\infty$ and direct integration. For $|x| \geq 2^{-k-m}$, the case $N=0$ is immediate from the support of $\upsilon_{k+m}$ and $\| \upsilon_{k+m} \|_1 \lesssim 1$, while for $N>0$, we write
\begin{equation*}
h_k \ast \upsilon_{k+m}(x)= \int \Big(h_k(x-y) - \sum_{j=0}^{N-1} h_k^{(j)}(x) \frac{(-y)^j}{j!} \Big) \upsilon_{k+m} (y) {{\text{\,\rm d}}} y.
\end{equation*}
Using Taylor's theorem and $|x-y|\sim |x|$ for $|x| \geq 2^{-k-m}$, one immediately obtains \eqref{eq:smofhk}.
By Young's convolution inequality \eqref{eq:smofhk} implies
\begin{equation}\label{eq:hkm-p-q}
\|h_k*\upsilon_{k+m} *g\|_q{\lesssim} 2^{(k+m) (\frac 1p-\frac 1q -\frac 12)} \|g\|_p
\end{equation}
for any function $g \in L^p$, where $1 \leq p \leq q \leq \infty$. This upgrades to a maximal estimate at intervals at scale $2^{-k-m}$. More precisely, for any interval $I \subset [1,2]$ of length $|I|=2^{-k-m}$ we have
\begin{equation}\label{eq:hkmmax-p-q}
\|\sup_{t \in I} |h_k*\upsilon_{k+m} *g(t-\cdot) |\|_q{\lesssim} 2^{(k+m) (\frac 1p-\frac 1q-\frac 12)} \|g\|_p.
\end{equation}
Indeed, by the vanishing moment conditions for $\upsilon_{k+m}'$, the inequality \eqref{eq:smofhk}, and consequently the inequality \eqref{eq:hkm-p-q}, also hold if we replace the function
$h_k*\upsilon_{k+m}$ with $2^{-k-m} h_k*\upsilon_{k+m}'$. One can then obtain \eqref{eq:hkmmax-p-q} from these by a standard Sobolev-type application of the fundamental theorem of calculus (see, for instance, \cite[Chapter XI, \S 3, Lemma 1]{Stein1993}).

We are now in a position to prove \eqref{eq:fixedkAm}. For each fixed $k \geq 0$, we tile $[1,2]$ into intervals of length $2^{-k}$ denoted by $J_{k,\mu}:=[\mu 2^{-k}, (\mu+1)2^{-k}]$ for each integer $2^k \leq \mu < 2^{k+1}$, and set $J_{k,\mu}^*:=[(\mu-2) 2^{-k}, (\mu+3)2^{-k}]$ to be the concentric interval to $J_{k,\mu}$ with 5 times its length. Since $h_k \ast \upsilon_{k+m}$ is supported on $(-2^{-k}, 2^{-k})$, we have that for $r \in I_k$ and $t \in J_{k,\mu}$,
\begin{equation}\label{eq:hkm support}
h_k \ast \upsilon_{k+m} \ast g (t-r) = h_k \ast \upsilon_{k+m} \ast [g{{\mathbbm 1}}_{J_{k,\mu}^*}] (t-r).
\end{equation}
For each fixed $J_{k,\mu}$ and $m \geq 1$, let $\{I_{k+m,\mu,\nu}\}_{\nu}$ be a minimal cover of $E \cap J_{k,\mu}$ by intervals of length $2^{-k-m}$, and note that there are $N(E \cap J_{k,\mu}, 2^{-k-m})$ of them.
Then we have, by \eqref{eq:hkmmax-p-q} and \eqref{eq:hkm support}, for $g:=\psi_{k+m} \ast f_0$,
\begin{align*}
&\| \sup_{t \in E} |{\mathscr{A}}_{k,m} f_0 (\cdot, t)| \|_{L^q(r \mathrm{d}r)} \\
& \qquad \lesssim 2^{k(\frac 12-\frac 1q)}
\Big( \sum_{\mu} \sum_\nu \int \sup_{t\in I_{k+m,\mu,\nu} }| h_k*\upsilon_{k+m} * [g{{\mathbbm 1}}_{ J_{k,\mu}^*}] (t-r) |^q {{\text{\,\rm d}}} r \Big)^{1/q} \\
& \qquad {\lesssim} 2^{k(\frac 12-\frac 1q)}
\Big( \sum_{\mu} N(E\cap J_{k,\mu}, 2^{-k-m}) 2^{(k+m)(\frac 1p-\frac 1q-\frac 12)q } \|g{{\mathbbm 1}}_{J_{k,\mu}^*}\|_p^q \Big)^{1/q}
\\& \qquad {\lesssim} 2^{k(\frac 1p-\frac 2q)} 2^{-m(\frac 12+\frac 1q-\frac 1p) } \sup_{|J|=2^{-k} }N(E\cap J, 2^{-k-m} ) ^{\frac 1q}
\Big(\sum_\mu \|g{{\mathbbm 1}}_{J_{k,\mu}^*}\|_p^p\Big)^{1/p}
\end{align*}
for $p \leq q$, and since the $J^*_{k,\mu}$ have bounded overlap, this yields the asserted inequality \eqref{eq:fixedkAm} in view of the definition of $\omega_{m}^{p,q}(E,k)$ in \eqref{eq:omega def}.
\end{proof}

\subsection{The operators \texorpdfstring{${{\mathfrak{M}}}_p^\pm$}{Mp}} The treatment shares similarities with \cite[Proposition 5.4]{SeegerWaingerWright1997}.
We break $\mathfrak{M}_p^\pm g \lesssim \mathfrak{M}_{p,0}^\pm g+ \mathfrak{M}_{p, \infty}^\pm g$, where
\begin{align*}
{{\mathfrak{M}}}_{p,0}^- g(r)&:=\sup_{\substack{ t\in E\\r/2<t<3r/2}}
r^{-1}\int_{|r-t|}^{2|r-t|}s^{ \frac 12- \frac 1p}(s-|r-t|)^{-\frac 12} |g(s)|{{\text{\,\rm d}}} s, \\
{{\mathfrak{M}}}_{p,\infty}^- g(r)&:=\sup_{\substack{ t\in E\\r/2<t<3r/2}}
r^{-1}\int_{2|r-t|}^{r+t}s^{ - 1/p} |g(s)|{{\text{\,\rm d}}} s,
\end{align*}
and with analogous definitions for ${{\mathfrak{M}}}_{p,0}^+$ and ${{\mathfrak{M}}}_{p,\infty}^+$ breaking the $s$-domain in the regions $(\frac{r+t}{2}, r+t)$ and $(|r-t|, \frac{r+t}{2})$, respectively.
The operators ${{\mathfrak{M}}}_{p,\infty}^\pm$ are pointwise bounded by the two-dimensional version of ${{\mathfrak{M}}}_p$ in \eqref{eq:Mp def}, so one can appeal to the bounds in \S\ref{sec:higher}. The main focus of this subsection is to study the operators ${{\mathfrak{M}}}_{p,0}^\pm$.

\begin{proposition}\label{prop:Mp-d=2}
Let $E \subset [1,2]$ and $1 \leq p\le q<\infty$.

\begin{enumerate}[(i)]
\item For $p>2$,
\begin{equation*}
\| {{\mathfrak{M}}}_{p}^\pm g \|_{L^q(r \mathrm{d}r)} \lesssim \|g \|_p.
\end{equation*}

\item For $p<2$,
\begin{equation}\label{eq:M0-p-q}
\| {{\mathfrak{M}}}_{p}^\pm g \|_{L^q(r \mathrm{d}r)}\lesssim_p \sup_{0 < \delta < 1} N(E, \delta )^{\frac 1q} \delta^{1-\frac 2p+\frac 1q}
\| g \|_{p}.
\end{equation}

\item For $p=2$,
\begin{equation}\label{eq:M0-2-q}
\| {{\mathfrak{M}}}_{p}^\pm g \|_{L^q(r \mathrm{d}r)} \lesssim \sum_{\ell\ge 0} (1+\ell) \sup_{n\ge \ell} N(E,2^{-n})^{1/q} 2^{-n/q}
\| g \|_{2}.
\end{equation}
\end{enumerate}
\end{proposition}

\begin{proof}
We use the decomposition ${{\mathfrak{M}}}_p^\pm g \lesssim {{\mathfrak{M}}}_{p,0}^\pm g + {{\mathfrak{M}}}_{p,\infty}^\pm g$, and for the operators ${{\mathfrak{M}}}_{p,\infty}^\pm$ the bounds in (i), (ii) and (iii) follow from Propositions \ref{lem:trivial holder}, \ref{lem:rightlineendpt} and \ref{lem:rightlineendpt=beta=1} respectively.

We shall then focus on ${{\mathfrak{M}}}_{p,0}:={{\mathfrak{M}}}_{p,0}^-$; the corresponding arguments for ${{\mathfrak{M}}}_{p,0}^+$ require only a minor notational modification. We observe that ${{\mathfrak{M}}}_{p,0} g(r)=0$ if $r\in {{\mathbb{R}}}\setminus [\frac 23,4]$ and hence $\| {{\mathfrak{M}}}_{p,0} g \|_{L^q(r \mathrm{d} r)} \sim \| {{\mathfrak{M}}}_{p,0} g \|_{L^q}$.

\medskip

\noindent \emph{Case $p >2$.} For this simpler case, we note the pointwise estimate
\begin{align*}
{{\mathfrak{M}}}_{p,0} g(r) & \lesssim \sum_{m \geq 0} \sup_{\substack{t \in [1,2] \\ r/2 < t < 3r/2}} \int_{|r-t|(1+2^{-m-1})}^{|r-t|(1+2^{-m})} s^{\frac{1}{2}-\frac{1}{p}} (s-|r-t|)^{-1/2} |g(s)| {{\text{\,\rm d}}} s \\
& \lesssim \sum_{m \geq 0} \sup_{\substack{t \in [1,2] \\ r/2 < t < 3r/2}} 2^{m/2} |r-t|^{-1/p} \int_{|r-t|(1+2^{-m-1})}^{|r-t|(1+2^{-m})} |g(s)| {{\text{\,\rm d}}} s.
\end{align*}
By H\"older's inequality, this is further estimated by a constant times
\begin{align*} &\sum_{m \geq 0} \sup_{\substack{t \in [1,2] \\ r/2 < t < 3r/2}} 2^{m/2} |r-t|^{-1/p} (|r-t|2^{-m})^{1/p'} \Big( \int_{0}^{6} |g(s)|^p {{\text{\,\rm d}}} s \Big)^{1/p} \\
& = \sum_{m \geq 0} 2^{-m(\frac{1}{2}-\frac{1}{p})} \sup_{\substack{t \in [1,2] \\ r/2 < t < 3r/2}} |r-t|^{1-\frac{2}{p}} \| g \|_p \lesssim \| g \|_p,
\end{align*}
since the $m$-sum converges for $p>2$.
The $L^q$-bound for ${{\mathfrak{M}}}_{p,0} $ then follows from trivial integration in the $r$-variable. This concludes the proof of (i).

\medskip

\noindent \emph{Case $p \leq 2$.} We treat (ii) and (iii) simultaneously. First, note that we can assume without loss of generaltiy that $|\overline{E}|=0$. Indeed, for $p<2$ the finiteness of the right-hand side of \eqref{eq:M0-p-q}
implies the estimate $\delta N(E,\delta){\lesssim} \delta^{-(1-\frac 2p)q} \to 0$ as $\delta \to 0$. Thus
$|\overline E|=0$. Similarly, the finiteness of the right-hand side of \eqref{eq:M0-2-q} also implies $|\overline E|=0$.

For each $n \geq 0$, let $D_n:=\{ r : 2^{-n} \leq {{\text{\rm dist}}} (r, E) < 2^{-n+1}\}$ and note that $|D_n| {\lesssim} N(E,2^{-n})2^{-n}$. Write $D_n = \bigcup_{\nu \in {{\mathfrak{N}}}_n} I_n^\nu$ where $I_n^\nu$ are disjoint intervals of length $|I_n^\nu|=2^{-n-1}$, with $\#{{\mathfrak{N}}}_n \sim N(E, 2^{-n})$. For each $\nu \in {{\mathfrak{N}}}_n$, and $0\leq \ell \leq n+1$, let $E_{n,\ell}^\nu:=\{t \in E : 2^{-n+\ell-1} \leq {{\text{\rm dist}}}(t, I_n^\nu) < 2^{-n+\ell}\}$. Then we can write
\begin{align*}
&{{\mathfrak{M}}}_{p,0} g(r) \lesssim \sum_{n \geq 0} \sum_{\nu \in {{\mathfrak{N}}}_n} {{\mathbbm 1}}_{I_n^\nu}(r) \sup_{0 \leq \ell \leq n+1} \sup_{t \in E_{n,\ell}^\nu} \int_{|r-t|}^{2|r-t|} \frac{s^{\frac{1}{2}-\frac{1}{p}}}{(s-|r-t|)^{1/2}} |g(s)| {{\text{\,\rm d}}} s \\
& \,\,\lesssim \sum_{\ell \geq 0} \sum_{m \geq 0} 2^{m/2} \sum_{\substack{n \in {\mathbb N}_0: \\ n \geq \ell -1}} \sum_{\nu \in {{\mathfrak{N}}}_n} {{\mathbbm 1}}_{I_n^\nu}(r) 2^{(n-\ell)/p} \sup_{t \in E_{n,\ell}^\nu} \int_{|r-t|(1+2^{-m-1})}^{|r-t|(1+2^{-m})} |g(s)| {{\text{\,\rm d}}} s \\
& \,\,=: \sum_{\ell \geq 0} \sum_{m \geq 0} {{\mathfrak{M}}}^{m,\ell} g(r).
\end{align*}
We now break the analysis depending on whether $m \leq \ell$ or $m > \ell$. If $m \leq \ell$, we have by Hölder's inequality
\begin{align*}
{{\mathfrak{M}}}^{m,\ell} g(r) {{\mathbbm 1}}_{I_n^\nu}(r) \lesssim 2^{m/2} 2^{(n-\ell)/p} \Big( \int_{2^{-n+\ell-1}}^{2^{-n+\ell+2}}|g(s)|^p {{\text{\,\rm d}}} s \Big)^{1/p} 2^{(-n+\ell-m)/p'}.
\end{align*}
Thus, we estimate $\|{{\mathfrak{M}}}^{m,\ell} g \|_{q}$ by a constant times
\begin{align*}
& \lesssim 2^{-\ell(\frac{2}{p}-1)} 2^{m(\frac{1}{p}-\frac{1}{2})} \Big( \sum_{ \substack{n \in {\mathbb N}_0 : \\ n \geq \ell-1}} |D_n| 2^{-n(1-\frac{2}{p})q} \Big( \int_{2^{-n+\ell-1}}^{2^{-n+\ell+2}}|g(s)|^p {{\text{\,\rm d}}} s \Big)^{\frac{q}{p}} \Big)^{\frac{1}{q}} \\
& \lesssim 2^{-\ell(\frac{2}{p}-1)} 2^{m(\frac{1}{p}-\frac{1}{2})} \sup_{n \geq \ell} N(E,2^{-n})^{1/q} 2^{-n(1+\frac{1}{q}-\frac{2}{p})} \| g \|_p
\end{align*}
using $ p \leq q$. Summing, we obtain for $p<2$
\begin{equation}\label{eq:m leq ell p<2}
\sum_{\ell\ge 0} \sum_{m\le \ell} \| {{\mathfrak{M}}}^{m,\ell} g \|_{q}
{\lesssim} (2-p)^{-2} \sup_{n \geq 0} N(E,2^{-n})^{1/q} 2^{-n(1+\frac{1}{q}-\frac{2}{p})} \|g\|_p,
\end{equation}
and, for $p=2$,
\begin{equation}\label{eq:m leq ell p=2}
\sum_{\ell\ge 0} \sum_{m\le \ell} \| {{\mathfrak{M}}}^{m,\ell} g \|_{q}
{\lesssim} \sum_{\ell \geq 0} (1+\ell) \sup_{n \geq \ell} N(E,2^{-n})^{1/q} 2^{-n/q} \|g\|_2.
\end{equation}

We next turn to the terms $m > \ell$ and perform a finer analysis. We further break the set $E_{n,\ell}^\nu$ into smaller intervals of length $2^{-n+\ell-m}$; we call these intervals $Q_{n,\ell,m}^{\nu,\rho}$, index them by $\rho$, and note there are $O(2^m)$ of them. Note that if $m > \ell$, these intervals are smaller than the $I_n^\nu$ intervals where the $r$-variable lives in.
We then have that for $r \in I_n^\nu$,
\begin{align*}
&\sup_{t \in E_{n,\ell}^\nu} \int_{|r-t| (1+2^{-m-1})}^{|r-t|(1+2^{-m})} |g(s)|{{\text{\,\rm d}}} s \lesssim \sup_{\rho} \sup_{t \in Q_{n,\ell,m}^{\nu, \rho}} \int_{|r -t| + 2^{-n+\ell-m-1}}^{|r-t|+2^{-n+\ell-m}} |g(s)| {{\text{\,\rm d}}} s \\
& \lesssim \sum_{\pm} \sup_{\rho } \int_{\pm (r - \tilde{Q}_{n,\ell,m}^{\nu,\rho})} |h(s)| {{\text{\,\rm d}}} s = \sup_\rho |h| \ast {{\mathbbm 1}}_{\tilde{Q}_{n,\ell.m}^{\nu,\rho}} (r) + \sup_\rho |\tilde{h}| \ast {{\mathbbm 1}}_{\tilde{Q}_{n,\ell.m}^{\nu,\rho}} (r)
\end{align*}
where $h(s):= g(s) {{\mathbbm 1}}_{\{2^{-n+\ell-1} \leq s \leq 2^{-n+\ell+2}\}}(s)$, $\tilde{h}(s)=h(-s)$ and $\tilde{Q}_{n,\ell,m}^{\nu,\rho}$ denotes the concentric triple of $Q_{n,\ell,m}^{\nu,\rho}$. Without loss of generality, we can assume the first term dominates, and using this bound in the definition of ${{\mathfrak{M}}}^{m,\ell}$ we obtain, using Young's convolution inequality,
\begin{align*}
& \| {{\mathfrak{M}}}^{m,\ell} g \|_{q} \lesssim 2^{m/2} \Big( \sum_{\substack{n \in {\mathbb N}_0: \\ n \geq \ell-1}} \sum_{\nu} \sum_\rho 2^{(n-\ell)\frac{q}{p}} \int_{I_n^\nu} \big| |h| \ast {{\mathbbm 1}}_{\tilde{Q}_{n,\ell,m}^{\nu,\rho}}(r)\big|^q {{\text{\,\rm d}}} r \Big)^{ \frac 1q} \\
& \lesssim 2^{m/2} \Big( \sum_{\substack{n \in {\mathbb N}_0: \\ n \geq \ell-1}} \sum_{\nu} \sum_\rho 2^{(n-\ell)\frac{q}{p}} |Q_{n,\ell,m}^{\nu,\rho}|^{q(1+\frac{1}{q}-\frac{1}{p})} \Big(\int_{2^{-n+\ell-1}}^{2^{-n+\ell+2}} |g(s)|^p {{\text{\,\rm d}}} s \Big)^{\frac qp} \Big)^{\frac 1q}
\end{align*}
which is
\begin{align*} &\lesssim 2^{-m(\frac{1}{2}+\frac{1}{q} - \frac{1}{p})} \Big( \sum_{\substack{n \in {\mathbb N}_0: \\ n \geq \ell-1}}
\frac{N(E,2^{-n+\ell-m})}{ 2^{(n-\ell)q(1+\frac{1}{q}-\frac{2}{p})} }\Big(\int_{2^{-n+\ell-1}}^{2^{-n+\ell+2}} |g(s)|^p {{\text{\,\rm d}}} s \Big)^{ \frac qp} \Big)^{\frac 1q} \\
& \lesssim 2^{-m(\frac{1}{2}+\frac{1}{q} - \frac{1}{p})} \sup_{n \geq \ell} N(E,2^{-n+\ell-m})^{1/q} 2^{(-n+\ell)(1+\frac{1}{q}-\frac{2}{p})} \| g \|_p
\end{align*}
where in the last step we used that $p \leq q$.
We next note that
\begin{align*}&\sum_{\ell\ge 0}\sum_{m> \ell} 2^{-m(\frac{1}{2}+\frac{1}{q} - \frac{1}{p})} \sup_{n \geq \ell} N(E,2^{-n+\ell-m})^{1/q} 2^{(-n+\ell)(1+\frac{1}{q}-\frac{2}{p})}
\\&=
\sum_{\ell\ge 0}\sum_{m> \ell} 2^{m(\frac 12-\frac 1p)}
\sup_{k\ge 0} N(E, 2^{-(k+m)}) 2^{-(k+m) (1-\frac 2p+\frac 1q) }
\\&{\lesssim} \sum_{m\ge 0} (1+m) 2^{m(\frac 12-\frac 1p)} \sup_{n\ge m} N(E, 2^{-n}) 2^{-n (1-\frac 2p+\frac 1q) }
\end{align*}
and consequently we obtain, for $p< 2$, \[\sum_{\ell\ge 0}\sum_{m> \ell}\|{{\mathfrak{M}}}^{m,\ell} g\|_q{\lesssim} (2-p)^{-2} \sup_{n\ge 0} N(E, 2^{-n}) 2^{-n (1-\frac 2p+\frac 1q) } \|g\|_p\]
and for $p=2$
\[\sum_{\ell\ge 0}\sum_{m> \ell}\|{{\mathfrak{M}}}^{m,\ell} g\|_q{\lesssim} \sum_{m \geq 0} (1+m) \sup_{n\ge m} N(E, 2^{-n}) 2^{-n/q } \|g\|_2.\]
Combining these with \eqref{eq:m leq ell p<2} and \eqref{eq:m leq ell p=2} concludes the proof of parts (ii) and (iii).
\end{proof}

\subsection{Proofs of Theorems \ref{thm:maintwodim} and \ref{thm:2dendpoint}}
We are now in a position to prove the $2$-dimensional results stated in the introduction.
We observe that for $d=2$ the triangle $\Delta_\beta$ is given by
\begin{equation}
\label{eq:Deltabeta set}
\Delta_\beta= \big\{ (\tfrac{1}{p}, \tfrac{1}{q}) \in [0,1]^2: \tfrac 1{2p}\le \tfrac 1q\le \tfrac{1}{p}, \,\,\, \tfrac{1-\beta} q\ge \tfrac 2p-1\big\}.
\end{equation}
We start with Theorem \ref{thm:maintwodim}, which gives a complete answer regarding $\overline{{{\mathcal{T}}}_E^{{\text{\rm rad}}}}$.

\begin{proof}[Proof of Theorem \ref{thm:maintwodim}]

The implication $\subset$ follows from the necessary conditions in Corollary \ref{cor: Delta beta nec} and Lemma \ref{lem: nec hard}.

We turn to the sufficient conditions. By Lemma \ref{lem:2d reduction}, it suffices to consider the operators $R_1^\pm, R_2^\pm$ and ${{\mathfrak{M}}}^\pm_p$. We first assume $\beta=\dim_M E < 1$.
\begin{enumerate}[(a)]
\item By Proposition \ref{prop:R1R2} and the definition of Minkowski dimension, we have that $R_1^\pm$ are $L^p(s \mathrm{d}s) \to L^q(r \mathrm{d}r)$ bounded if $\frac{1-\beta}{q} + \frac{1}{2} -\frac{1}{p}>0$. Note that this only constitutes a constraint for $p < 2$, which is subdominant with respect to the condition in \eqref{eq:Deltabeta set}. Thus, bounds for $R_1^\pm$ hold when $\beta<1$ and $(\frac 1p, \frac 1q)\in \Delta_\beta$.

\item By the definition of $\nu^\#$ in \eqref{eq:nudagger}, we have for any $\nu>\nu^\#(\frac{q}{2}-1)$
\begin{equation}\label{eq:nu no dagger}
\sup_{|J|=2^{-k}} |J|^{1- \frac{q}{2}} N(E \cap J, \delta) \lesssim \delta^{-\nu}
\end{equation}
for all $0 < \delta \leq 2^{-k}$, with implicit constant independent of $k$. Recalling the definition of $\omega_{m}^{p,q}(E,k)$ in \eqref{eq:omega def}, we get from \eqref{eq:nu no dagger}
\begin{equation*}
2^{-m(\frac{1}{2}+ \frac{1}{q} - \frac{1}{p})} \omega_m^{p,q}(E,k) \lesssim 2^{(m+k)(\frac{1}{p}-\frac{1}{2} - \frac{1-\nu}{q})}.
\end{equation*}
Thus, by Proposition \ref{prop:R3-better}, (i), and choosing a suitable $\nu$ we see that $R_2^\pm$ is $L^p(s \mathrm{d}s) \to L^q(r \mathrm{d}r)$ bounded if $\frac{1}{p}-\frac{1}{2} < \frac{1}{q} ( 1-\nu^\#(\frac{q}{2}-1))$ and $2 \leq q\leq 2p$. For $q <2$, we directly use the definition of Minkowski dimension in $\omega_{m}^{p,q}(E,k)$ and obtain that, for every $\varepsilon>0$,
\begin{align*}
& 2^{-m(\frac{1}{2} + \frac{1}{q}-\frac{1}{p})} \omega_m^{p,q}(E,k)\\
& \qquad \qquad \lesssim_\varepsilon 2^{-m(\frac{1}{p}- \frac{1}{2} - \frac{\varepsilon}{2})} 2^{-k(\frac{1}{p} + \frac{1}{q} - 1 - \frac{\varepsilon}{q})} 2^{-(m+k)(\frac{1-\beta}{q} - \frac{2}{p} +1)}.
\end{align*}
This is further bounded by $2^{-m(\frac{1}{p}- \frac{1}{2} - \frac{\varepsilon}{2})} 2^{-k(\frac{1}{p} + \frac{1}{q} - 1 - \frac{\varepsilon}{q})}$ if $(\frac{1}{p}, \frac{1}{q}) \in \Delta_\beta$. Furthermore, these exponents are negative if $p \leq q < 2$, provided $\varepsilon>0$ is chosen sufficiently small. Thus, by Proposition \ref{prop:R3-better}, (ii) gives that $R_2^\pm$ is $L^p(s \mathrm{d}s) \to L^q(r \mathrm{d}r)$ bounded if $1 \leq p \leq q < 2$ and $(\frac{1}{p}, \frac{1}{q}) \in \Delta_\beta$ for $\beta <1$. Consequently, we have shown the inclusion $\supset$ for $R_2^\pm$ if $\beta <1$.

\item By Proposition \ref{prop:Mp-d=2} and the definition of Minkowski dimension, the operators ${{\mathfrak{M}}}^\pm_p$ are $L^p \to L^q(r \mathrm{d}r)$ bounded for $2 \leq p \leq q$. For $1 \leq p < 2$, we have boundedness if $\frac{1-\beta}{q} + 1 - \frac{2}{p} > 0$. Thus, in view of \eqref{eq:Deltabeta set}, we have the
${{\mathfrak{M}}}_p^\pm$ bounds for $(\frac{1}{p}, \frac{1}{q}) \in \mathrm{int}(\Delta_\beta)$.

\end{enumerate}

For the case $\beta=1$, we use $N(E,\delta) \lesssim \delta^{-1}$ for any $0 < \delta < 1$. If $(\frac{1}{p}, \frac{1}{q}) \in \mathrm{int} (\Delta_1)$ we have $p>2$. The desired bounds for $R_1^\pm$ and ${{\mathfrak{M}}}_p^\pm$ follow immediately from Propositions \ref{prop:R1R2} and \ref{prop:Mp-d=2}, (i). For the operator $R_2^\pm$, the argument in (b) above yields that it is $L^p(s \mathrm{d}s) \to L^q(r \mathrm{d}r)$ bounded if $2 < p \leq q < 2p$ and $\frac{1}{p}-\frac{1}{2} < \frac{1}{q}(1-\nu^\#(\frac{q}{2}-1))$. This completes the proof.
\end{proof}

We conclude by giving the proof of Theorem \ref{thm:2dendpoint}, which addresses endpoint situations in 2 dimensions.

\begin{proof}[Proof of Theorem \ref{thm:2dendpoint}]
Recall from Corollary \ref{cor: Delta beta nec} that ${{\mathcal{T}}}_E^{{\text{\rm rad}}} \subset \Delta_\beta$. We first note the following:
\begin{enumerate}[(a)]
\item If $\beta <1$, we have shown in the proof of Theorem \ref{thm:maintwodim}, (a), that $R_1^\pm :L^p(s\mathrm{d}s) \to L^q(r\mathrm{d}r)$ for $(\frac{1}{p},\frac{1}{q}) \in \Delta_\beta$.
\item By the definition of $\omega_{m}^{p,q}(E,k)$ in \eqref{eq:omega def}, we get for all $\varepsilon>0$
\begin{equation*}
\omega_{m}^{p,2p}(E,k) \lesssim_\varepsilon 2^{m(\gamma+\varepsilon)/2p}.
\end{equation*}
Using this in Proposition \ref{prop:R3-better}, (i), yields that $R_2^\pm$ is $L^{p}(s \mathrm{d}s) \to L^{2p}(r \mathrm{d}r)$ bounded if $p > 1+\gamma$.
\item If $1 \leq p \leq q < \infty$ and $\frac{1-\beta}{q}+1-\frac{2}{p}>0$, we have by Proposition \ref{prop:Mp-d=2} that ${{\mathfrak{M}}}_p^\pm$ is $L^p \to L^q(r \mathrm{d}r)$ bounded. If
$\sup_{\delta<1} \delta^\beta N(E,\delta) < \infty$ and $\beta <1$, Proposition \ref{prop:Mp-d=2} implies that ${{\mathfrak{M}}}_p^\pm$ is $L^p \to L^q(r \mathrm{d}r)$ bounded for all $(\frac{1}{p}, \frac{1}{q}) \in \Delta_\beta$.
\end{enumerate}
In view of (a), the endpoint bounds in parts (i)-(iv) of Theorem \ref{thm:2dendpoint} are only determined by the operators $R_2^\pm$ and ${{\mathfrak{M}}}_p^\pm$. We now discuss item by item.
\begin{enumerate}[(i)]
\item Assume $2 \gamma - \beta < 1$; this implies $\beta < 1$.
For the left to right implication, note that if ${{\mathcal{T}}}_E^{{{\text{\rm rad}}}}=\Delta_\beta$, we have that $L^p_{\mathrm{rad}} \to L^q$ bounds for $M_E$ hold on the line $\frac{1-\beta}{q}=\frac{2}{p}-1$: see \eqref{eq:Deltabeta set}. But this implies, by Lemma \ref{lem: nec easy}, (ii), the claimed necessary condition $\sup_{\delta<1} \delta N(E,\delta)< \infty$.

For the reverse implication, we have by Lemma \ref{lem:R2-p=q} that $R_2^\pm$ is bounded on $L^{1+\beta}(r \mathrm{d}r)$, which corresponds to the point $P_{2,\beta}$. By interpolation, it suffices to show $L^{p}(s \mathrm{d}s) \to L^{2p}(r \mathrm{d}r)$ boundedness for $(\frac{1}{p}, \frac{1}{q})\in [P_1, P_{3,\beta}^{\mathrm{rad}}]$, that is, for $p \geq \frac{3+\beta}{2}$. But since $2 \gamma - \beta < 1$, this follows by (b) above. This implies $\Delta_\beta \subset {{\mathcal{T}}}_E^{{\text{\rm rad}}}$.

\item\label{item2} Assume $2 \gamma - \beta = 1$, $\beta <1$ and $\sup_\delta \delta^\beta N(E,\delta) < \infty$.
If $\beta=0$, we have by assumption that $N(E,\delta)=O(1)$ uniformly in $\delta$. The bounds for $M_E$ are then an immediate consequence of the classical result of \cite{Littman} on a single spherical average, which also includes the point $P_3=(\frac{2}{3}, \frac{1}{3})$.

We can therefore assume, in what follows, $0 < \beta <1$. By (a) and (c), it suffices to consider $R_2^\pm$.
Since $1+\gamma=\frac{3+\beta}{2}$, the argument in (b) only gives boundedness for $(\frac{1}{p},\frac{1}{q}) \in [P_1, P_{3,\beta}^{\mathrm{rad}})$. However, we have by Lemma \ref{lem:R2-p=q} that $R_2^\pm$ is $L^p(r\mathrm{d}r)$-bounded if either $\beta<1/2$ and $p\ge 1$, or $\beta\ge 1/2$ and $p>2\beta$. Thus, if $(\frac 1p,\frac 1q)\in (P_{2,\beta}, P_{3,\beta}^{{\text{\rm rad}}}$) (hence $\frac{1-\beta}q=\frac 2p-1$), we can find $p_0$ with $\max\{1,2\beta\}<p_0<1+\beta$ and $p_1>1+{\gamma}$ with $(\frac{1}{p_1},\frac{1}{2p_1})\in (P_1, P_{3,\beta}^{{\text{\rm rad}}})$ such that \[ (\tfrac 1p,\tfrac 1q)= (1-{\vartheta}) (\tfrac{1}{p_0}, \tfrac{1}{p_0} )+{\vartheta} (\tfrac{1}{p_1},\tfrac{1}{2p_1})\] for some $0<{\vartheta}<1$, and with $R_2^\pm$ bounded on $L^{p_0}(r\mathrm{d}r)$ and from $L^{p_1}(s\mathrm{d}s) \to L^{2p_1}(r\mathrm{d}r)$ bounded.
Hence, $R_2^\pm$ maps $L^p(s\mathrm{d} s)$ to $L^q(r\mathrm{d} r)$ by interpolation. This yields the desired inclusion $\Delta_\beta \backslash \{P_{3,\beta}^{{\text{\rm rad}}}\} \subset {{\mathcal{T}}}_E^{{\text{\rm rad}}}$.

\item Assume $2 \gamma - \beta =1$ and $\sup_\delta \delta^\beta N(E,\delta) = \infty$ (note this means $\beta <1$). Since the line joining $P_{2,\beta}$ and $P_{3,\beta}^{\mathrm{rad}}$ is given by $\frac{1-\beta}{q}=\frac{2}{p}-1$, we have by Lemma \ref{lem: nec easy}, (ii), that $M_E$ is not $L^p_{\mathrm{rad}} \to L^q$ bounded for $(\frac{1}{p}, \frac{1}{q}) \in [P_{2,\beta}, P_{3,\beta}^{\mathrm{rad}}]$.
The bounds on $R_2^\pm$ are as in \eqref{item2}. Furthermore,
if $\frac{1-\beta}{q}>\frac{2}{p}-1$, the item (c) above guarantees $L^p \to L^q(r \mathrm{d}r)$ bounds for ${{\mathfrak{M}}}_p^\pm$. Thus, ${{\mathcal{T}}}_E^{{\text{\rm rad}}}=\Delta_\beta \backslash [P_{2,\beta}, P_{3,\beta}^{{{\text{\rm rad}}}}]$.

\item Assume $2\gamma-\beta > 1$; this implies $\beta <1$. Since $P_{4,\gamma}^{{{\text{\rm rad}}}}=(\frac{1}{1+\gamma}, \frac{1}{2(1+\gamma)})$, the claimed bounds for $R_2^\pm$ on $[P_1,P_{4,\gamma}^{{\text{\rm rad}}})$ follow from (b) above. The boundedness of ${{\mathfrak{M}}}_p^\pm$ on this line segment follows from (c) above, since $P_{4,\gamma}^{{\text{\rm rad}}}$ lies in the line segment $(P_1, P_{3,\beta}^{{\text{\rm rad}}})$.

\item Assume $\beta=1$. Then the line segment $[P_{2}, P_3^{{\text{\rm rad}}}]$ is vertical with $p=2$, and the condition $\sup_{\delta<1} \delta \log(\frac 1\delta)N(E,\delta)=\infty$ and Lemma \ref{lem:lowerbd-pd} exclude $L^2_{{\text{\rm rad}}} \to L^q$ boundedness for any $q\ge 2$. On the other hand, if $p>2$ we have by Proposition \ref{prop:R1R2} that $R_1^\pm$ is $L^p(s \mathrm{d}s) \to L^q(r \mathrm{d}r)$ for all $2 < p \leq q \leq \infty$, with analogous bounds for ${{\mathfrak{M}}}_p^\pm$ by Proposition \ref{prop:Mp-d=2}, (i). For $R_2^\pm$, the bounds for $2 < p \leq q \leq 2p$ follow from (b) above. Consequently, $M_E$ is $L^p_{{\text{\rm rad}}} \to L^q$ bounded if and only if $2 < p \leq q \leq 2p$.
\end{enumerate}
This concludes the proof.
\end{proof}

\end{document}